\documentclass[12pt]{amsart}
\pdfoutput=1 
\usepackage[dvips]{graphicx}
\usepackage{amssymb, amstext, amscd, amsmath, color}
\usepackage{epstopdf}
\usepackage{tikz,mathdots,enumerate}
\usepackage{pgfplots}
\usepackage{url}
\usepackage{tikz-cd}
\usetikzlibrary{arrows}
\usetikzlibrary{decorations.markings}

\definecolor{cof}{RGB}{219,144,71}
\definecolor{pur}{RGB}{186,146,162}
\definecolor{greeo}{RGB}{91,173,69}
\definecolor{greet}{RGB}{52,111,72}

\usepackage{hyperref}
\usepackage{hhline}

\begin{document}

\newtheorem{theorem}{Theorem}[section]
\newtheorem{corollary}[theorem]{Corollary}
\newtheorem{proposition}[theorem]{Proposition}
\newtheorem{lemma}[theorem]{Lemma}

\theoremstyle{definition}
\newtheorem{remark}[theorem]{Remark}
\newtheorem{definition}[theorem]{Definition}
\newtheorem{example}[theorem]{Example}
\newtheorem{conjecture}[theorem]{Conjecture}

\newcommand{\FFock}{\mathcal{F}}
\newcommand{\kil}{\mathsf{k}}
\newcommand{\Hil}{\mathsf{H}}
\newcommand{\hil}{\mathsf{h}}
\newcommand{\Kil}{\mathsf{K}}
\newcommand{\Real}{\mathbb{R}}
\newcommand{\Rplus}{\Real_+}

\newcommand{\bC}{{\mathbb{C}}}
\newcommand{\bD}{{\mathbb{D}}}
\newcommand{\bN}{{\mathbb{N}}}
\newcommand{\bQ}{{\mathbb{Q}}}
\newcommand{\bR}{{\mathbb{R}}}
\newcommand{\bT}{{\mathbb{T}}}
\newcommand{\bX}{{\mathbb{X}}}
\newcommand{\bZ}{{\mathbb{Z}}}
\newcommand{\bH}{{\mathbb{H}}}
\newcommand{\BH}{{\B(\H)}}
\newcommand{\bsl}{\setminus}
\newcommand{\ca}{\mathrm{C}^*}
\newcommand{\cstar}{\mathrm{C}^*}
\newcommand{\cenv}{\mathrm{C}^*_{\text{env}}}
\newcommand{\rip}{\rangle}
\newcommand{\ol}{\overline}
\newcommand{\td}{\widetilde}
\newcommand{\wh}{\widehat}
\newcommand{\sot}{\textsc{sot}}
\newcommand{\wot}{\textsc{wot}}
\newcommand{\wotclos}[1]{\ol{#1}^{\textsc{wot}}}
 \newcommand{\A}{{\mathcal{A}}}
 \newcommand{\B}{{\mathcal{B}}}
 \newcommand{\C}{{\mathcal{C}}}
 \newcommand{\D}{{\mathcal{D}}}
 \newcommand{\E}{{\mathcal{E}}}
 \newcommand{\F}{{\mathcal{F}}}
 \newcommand{\G}{{\mathcal{G}}}
\renewcommand{\H}{{\mathcal{H}}}
 \newcommand{\I}{{\mathcal{I}}}
 \newcommand{\J}{{\mathcal{J}}}
 \newcommand{\K}{{\mathcal{K}}}
\renewcommand{\L}{{\mathcal{L}}}
 \newcommand{\M}{{\mathcal{M}}}
 \newcommand{\N}{{\mathcal{N}}}
\renewcommand{\O}{{\mathcal{O}}}
\renewcommand{\P}{{\mathcal{P}}}
 \newcommand{\Q}{{\mathcal{Q}}}
 \newcommand{\R}{{\mathcal{R}}}
\renewcommand{\S}{{\mathcal{S}}}
 \newcommand{\T}{{\mathcal{T}}}
 \newcommand{\U}{{\mathcal{U}}}
 \newcommand{\V}{{\mathcal{V}}}
 \newcommand{\W}{{\mathcal{W}}}
 \newcommand{\X}{{\mathcal{X}}}
 \newcommand{\Y}{{\mathcal{Y}}}
 \newcommand{\Z}{{\mathcal{Z}}}

\newcommand{\supp}{\operatorname{supp}}
\newcommand{\conv}{\operatorname{conv}}
\newcommand{\cone}{\operatorname{cone}}
\newcommand{\vspan}{\operatorname{span}}
\newcommand{\proj}{\operatorname{proj}}
\newcommand{\sgn}{\operatorname{sgn}}
\newcommand{\rank}{\operatorname{rank}}
\newcommand{\Isom}{\operatorname{Isom}}
\newcommand{\qIsom}{\operatorname{q-Isom}}
\newcommand{\Cknet}{{\mathcal{C}_{\text{knet}}}}
\newcommand{\Ckag}{{\mathcal{C}_{\text{kag}}}}
\newcommand{\rind}{\operatorname{r-ind}}
\newcommand{\lind}{\operatorname{r-ind}}
\newcommand{\ind}{\operatorname{ind}}
\newcommand{\coker}{\operatorname{coker}}
\newcommand{\ran}{\operatorname{ran}}
\newcommand{\Aut}{\operatorname{Aut}}
\newcommand{\Hom}{\operatorname{Hom}}
\newcommand{\GL}{\operatorname{GL}}
\newcommand{\tr}{\operatorname{tr}}

\newcommand\blue{\color{blue}}
\newcommand\red{\color{red}}
\definecolor{dark_purple}{rgb}{0.4, 0.0, 0.4}
\definecolor{dark_green}{rgb}{0.0, 0.7, 0.0}
\newcommand\green{\color{dark_green}}
\newcommand\black{\color{black}}

\newcommand{\eqnwithbr}[2]{%
\refstepcounter{equation}
\begin{trivlist}
\item[]#1 \hfill $\displaystyle #2$ \hfill (\theequation)
\end{trivlist}}

\tikzset{->-/.style={decoration={
  markings,
  mark=at position .53 with {\arrow{>}}},postaction={decorate}}}

\setcounter{tocdepth}{1}

\title[Symbol functions for symmetric frameworks]{Symbol functions for symmetric frameworks}

\author[E. Kastis]{Eleftherios Kastis}
\email{l.kastis@lancaster.ac.uk}
\address{Dept.\ Math.\ Stats.\\ Lancaster University\\
Lancaster, LA1 4YF \\U.K. }

\author[D. Kitson]{Derek Kitson}
\email{derek.kitson@mic.ul.ie}
\address{Dept.\ Math.\ Comp. St.\\Mary Immaculate College, Thurles, Co.~Tipperary, Ireland.}

\author[J.E. McCarthy]{John E. M\raise.5ex\hbox{c}Carthy}
\email{mccarthy@wustl.edu}
\address{Dept.\ Math.\ Stats.\\ Washington University \\
St Louis, MO 63130\\U.S.A. }

\thanks{E.K. and D.K. supported by the Engineering and Physical Sciences Research Council [grant number EP/S00940X/1].
J.E.M. partially supported by National Science Foundation Grant  
DMS 156243}


\begin{abstract}
We prove a variant of the well-known result that intertwiners for the bilateral shift on $\ell^2(\bZ)$ are unitarily equivalent to multiplication operators on $L^2(\bT)$. This enables us to unify and extend fundamental aspects of  rigidity theory for bar-joint frameworks with an abelian symmetry group. In particular, we formulate the symbol function for a wide class of frameworks and show how to construct generalised rigid unit modes in a variety of new contexts. 
\end{abstract}

\maketitle
\tableofcontents


\section{Introduction}
A {\em bar-joint framework} in $d$-dimensional Euclidean space $\bR^d$ is a pair $(G,p)$ where $G=(V,E)$ is a simple undirected graph and $p\in (\bR^d)^V$ is an assignment of points in $\bR^d$ to each of the vertices in $G$. 
The edges of this embedded graph can be viewed as rigid bars of fixed length and the vertices as rotational joints. 
Such models arise naturally in engineering and the natural sciences  in contexts where their rigidity and flexibility properties are of particular interest (eg.~structural engineering \cite{max}, mineralogy \cite{gdph}, protein analysis \cite{gas-cse}, network localisation \cite{asp} and formation control \cite{kbf}). In this article we continue the recent development of operator theoretic methods for the analysis of infinitesimal (i.e.~first-order) flexibility in bar-joint frameworks (and other related frameworks). This line of research was initiated in Owen and Power (\cite{owe-pow}). (See also \cite{bkp,kas-kit-pow,pow,pow-poly}.) 

The presence of an infinitesimal flex can sometimes be explained by an inherent symmetry in the bar-joint framework and in recent years this interplay between symmetry and rigidity has received considerable attention (\cite{con,gue-fow-pow}).   
For example, it is well-known that the rigidity matrix $R(G,p)$ for a finite bar-joint framework with an abelian symmetry group admits a block-diagonalisation over the irreducible representations of the group. Moreover, the diagonal blocks can be described explicitly by associated {\em orbit matrices}. This property has been utilised to obtain combinatorial characterisations of so-called {\em forced} and {\em incidental} rigidity for finite bar-joint frameworks in dimension $2$. (See \cite{jkt,schtan}.)  

Periodic bar-joint frameworks have also received much attention in recent years. Here $R(G,p)$ is an infinite matrix and so operator theory naturally  comes to the fore. In \cite{owe-pow}, it is shown that the rigidity matrix for a periodic bar-joint framework gives rise to a Hilbert space operator which is unitarily equivalent to a multiplication operator $M_\Phi$. The symbol function $\Phi$ is matrix-valued and defined on the $d$-torus $\bT^d$. The set of points in $\bT^d$ where $\Phi$ has a non-zero kernel is known as the {\em RUM spectrum} and takes its name from the phenomenon of {\em rigid unit modes} (RUMs) in  silicates and zeolites (see \cite{dove,dhh,gdph}). 
   
RUM theory for periodic bar-joint frameworks and the aforementioned decomposition theory for finite bar-joint frameworks can be viewed as two sides of the same coin. The first aim of this article is to formalise this viewpoint using techniques from Fourier analysis. The second aim is to extend the theory so that it may be applied in new contexts. 

In Section \ref{s:intertwining}, we prove a variant of the well-known result that intertwiners for the bilateral shift on $\ell^2(\bZ)$ are unitarily equivalent to multiplication operators on $L^2(\bT)$ (Theorem \ref{t:intertwiningtwist}). The distinguishing features of our theorem are that it takes place in the setting of a general locally compact abelian group, with vector-valued function spaces, and in the presence of an additional {\em twist} arising from a unitary representation.   

In Section \ref{s:symbolfunctions}, we adopt the approach taken in \cite{kas-kit-pow} and introduce the more general notions of a {\em framework} $(G,\varphi)$ for a pair of Hilbert spaces $X$ and $Y$ and an accompanying {\em coboundary matrix} $C(G,\varphi)$. This convention simplifies the proofs and also allows the results to be applied in a much wider variety of settings (as demonstrated in the final section). 
Applying the results of Section \ref{s:intertwining}, we show that a framework with a discrete abelian symmetry group gives rise to a Hilbert space coboundary operator $C(G,\varphi)$ which admits a factorisation as illustrated in Figure \ref{f:factorisation} (Theorem \ref{t:rigidityoperator}).
Note that the block diagonalisation result for finite bar-joint frameworks and the unitary equivalence result for periodic bar-joint frameworks described above both follow from this factorisation.
We then provide an explicit description of the associated symbol function $\Phi$ in terms of generalised orbit matrices (Theorem \ref{t:symbol}) and as a trigonometric polynomial (Corollary \ref{cor:FCoef}).
\begin{figure}
\label{f:factorisation}
\begin{tikzcd}[row sep=3.5em,column sep=7.4em]
\ell^2(V,X) \arrow{r}{C(G,\varphi)} \arrow[swap]{d}{S_V} 
& \ell^2(E,Y)  \\%
\ell^2(\Gamma,X^{V_0}) \arrow[dotted,darkgray]{r}{\tilde{C}(G,\varphi)} \arrow[swap]{d}{F_{X^{V_0}}} \arrow[out=175,in=185,loop,swap, "T_{\tilde{\tau}}"] 
&  \ell^2(\Gamma,Y^{E_0}) \arrow[swap]{u}{S_E^{-1}}\\
L^2(\hat{\Gamma},X^{V_0}) \arrow{r}{M_{\Phi}} 
& L^2(\hat{\Gamma},Y^{E_0})\arrow[swap]{u}{F^{-1}_{Y^{E_0}}} 
\end{tikzcd}
\caption{Factorisation of the $\ell^2$-coboundary operator $C(G,\varphi)$ for a  framework $(G,\varphi)$ with a discrete abelian symmetry group $\Gamma$.}
\end{figure}

In Section \ref{s:RUM}, we introduce a generalised RUM spectrum $\Omega(\G)$ for frameworks with a discrete abelian symmetry group $\Gamma$ and show how to construct $\chi$-symmetric vectors $z(\chi,a)$ which lie in the kernel of the coboundary matrix $C(G,\varphi)$ for each $\chi\in\Omega(\G)$ (Theorem \ref{t:twistedflex}). Note that here we continue to work in the more general setting of coboundary operators and that the RUM spectrum is presented as a subset of the dual group $\hat{\Gamma}$. In the terminology of \cite{dove,dhh,gdph}, characters $\chi\in\hat{\Gamma}$ correspond to wave-vectors in reciprocal space and $\chi$-symmetric vectors which lie in the kernel of $C(G,\varphi)$ correspond to generalised rigid unit modes.

Finally, in Section \ref{s:examples}, we illustrate the results of the preceding sections with several contrasting examples. These include a bar-joint framework in $\bR^3$ with screw axis symmetry, a direction-length framework in $\bR^2$ with both translational and reflectional symmetry and a symmetric bar-joint framework in $\bR^3$ with mixed-norm distance constraints. For each example, we provide some necessary background, formulate the symbol function $\Phi$, compute the RUM spectrum $\Omega(\G)$ and construct generalised rigid unit modes $z(\chi,a)$ for points $\chi\in\Omega(\G)$. To the best of our knowledge, the interplay between rigidity and symmetry has not previously been explored in these contexts.

\section{Intertwining relations}
\label{s:intertwining}
Let $\Gamma$ be a locally compact Hausdorff abelian group.
Denote by $L^2(\Gamma)$ the Hilbert space of square integrable functions, i.e. Borel-measurable functions $f:\Gamma\to \bC$ such that,
\[\int_\Gamma |f(\gamma)|^2 \,d\gamma<\infty\]
where we use normalised Haar measure on $\Gamma$. Recall the Haar measure of a locally compact group is decomposable on $\Gamma$; in particular, $\Gamma$ contains a $\sigma$-compact clopen subgroup (\cite{fol}).

\subsection{The scalar case}
Given a set $\S$ of bounded operators on a Hilbert space $\H$, recall that its commutant is the unital $w^*$-closed algebra 
\[\S'=\{T\in B(\H)\,:\,TS=ST, \text{ for all } S\in \S\}.\] 
If $\S$ is a selfadjoint set, i.e.~$S^*\in \S$ for all $S\in \S$, then $\S'$ is also selfadjoint and hence a $C^*$-algebra. Moreover, $\S$ is a set of commuting operators if and only if  $\S\subseteq \S'$. Thus, an operator set is \emph{maximal abelian} if and only if $\S=\S'$ (\cite{mur}). 

\begin{proposition}\label{p:masa}
The algebra of multiplication operators $\M_\mu=\{M_f\,:\, f\in L^\infty(\Gamma)\}$ is a maximal abelian selfadjoint subalgebra of $B(L^2(\Gamma))$.
\end{proposition}

\begin{proof}
$\M_\mu$ is abelian, so $\M_\mu$ is a subset of its commutant. For the reverse inclusion, let $T\in (\M_\mu)'$.
We shall show that there exists $g\in L^\infty(\Gamma)$, such that $T=M_g$.
\begin{enumerate}[(i)]
\item Suppose first that $\Gamma$ is compact, so $\mu(\Gamma)<\infty$. Then the constant function $1_\Gamma$ lies in $L^2(\Gamma)$. Define $g=T1_\Gamma \in L^2(\Gamma).$ Then for every $f\in L^\infty(\Gamma)$, we have
\[Tf=T(f1_\Gamma)=TM_f 1_\Gamma= M_f T1_\Gamma=M_f g=fg=gf.\]
Hence, it suffices to show that $g\in L^\infty(\Gamma)$. Let $\alpha >0$ and $\Gamma_\alpha=\{\gamma\in \Gamma: |g(\gamma)|>\alpha\}$. Let $1_\alpha$ be the characteristic function of $\Gamma_\alpha$. Then
\[\|T1_\alpha\|_2^2=\int_\Gamma |g 1_\alpha|^2 d\mu=\int_{\Gamma_\alpha} |g|^2 d\mu\geq \alpha^2 \mu(\Gamma_\alpha)=\alpha^2 \|1_\alpha\|_2^2,\]
hence $\alpha\leq \|T\|$ whenever $\mu(\Gamma_\alpha)>0$. Thus $\|g\|_\infty\leq \|T\|$.
\item Suppose now that $\Gamma$ is $\sigma$-compact. Then $\Gamma$ can be written as a countable union of pairwise disjoint precompact sets $\Gamma_n$. Write $1_n$ for the characteristic function of $\Gamma_n$ and let $g_n=T1_n$. Similarly
to the previous case, we obtain that $TM_{1_n}=M_{g_n}$ and $\|g_n\|_\infty\leq \|T\|$ for every $n\in \mathbb{N}$. Hence define $g\in L^\infty(\Gamma)$ by $g\big|_{\Gamma_n}=g_n\big|_{\Gamma_n}$, for every $n\in \mathbb{N}$. Then $\|g\|_\infty \leq \sup_n \|g_n\|_\infty\leq \|T\|$, so $g\in L^\infty(\Gamma)$, and for every $f\in L^2(\Gamma)$ we have
\[M_g f=\sum\limits_{n=1}^\infty M_{1_n} M_g f = \sum\limits_{n=1}^\infty  M_{g_n}f= \sum\limits_{n=1}^\infty  T M_{1_n}f= \sum\limits_{n=1}^\infty M_{1_n} T f= T f.\]
(Each of the infinite sums should be interpreted as limits in $L^2$ of the partial sums.)

\item In the general case, let $H$ be a clopen $\sigma$-compact subgroup of $\Gamma$ and let $Z$ be a subset of $\Gamma$ that contains exactly one element of each coset of $H$, so that $\Gamma$ can be written as the disjoint union of the sets $z+H,\, z\in Z$. For each $z\in Z$, denote by $1_z$ the characteristic function of $z+H$ and let $g_z=T1_z$.
Similarly to the above cases, we have $TM_{1_z}=M_{g_z}$ and $\|g_z\|_\infty\leq \|T\|$ for every $z\in Z$. Define $g\in L^\infty(\Gamma)$ by $g\big|_{z+H}=g_z\big|_{z+H}$, for every $z\in Z$. Then $g$ is locally almost everywhere well-defined, $\|g\|_\infty \leq \sup_z \|g_z\|_\infty\leq \|T\|$, so $g\in L^\infty(\Gamma)$. Now given any function $f\in L^2(\Gamma)$, there exists a countable family $\{z_n\,:\,n\in \mathbb{N}\}\subseteq Z$ such that the set $\supp(f)\cap (\Gamma\backslash (\cup_{n}\, z_n+H))$ is null (\cite[Appendix E8]{rud}). Check that since $T$ commutes with the multiplication operators of characteristic functions, it follows that $\supp(Tf)\subseteq \supp(f)$.
 Hence
\[M_g f=\sum\limits_{n=1}^\infty M_{1_{z_n}} M_g f = \sum\limits_{n=1}^\infty  M_{g_{z_n}}f= \sum\limits_{n=1}^\infty  T M_{1_{z_n}}f =\sum\limits_{n=1}^\infty M_{1_{z_n}} T f= T f.\]

\end{enumerate}
\end{proof}

The   Fourier transform $F: (L^1\cap L^2)(\Gamma)\rightarrow L^2(\hat{\Gamma})$ given by the formula   
\[\hat{f}(\xi) = \int_{\Gamma} \overline{\xi(\gamma)}f(\gamma)d\gamma\]
extends uniquely to a unitary isomorphism from $L^2(\Gamma)$ to $L^2(\hat{\Gamma})$ (\cite{fol, rud}).
The inverse Fourier transform of a function $f\in L^2(\hat{\Gamma})$ is denoted $\check{f}$.

For each $\gamma\in \Gamma$, denote by $D_{\gamma}$ the unitary operator
\[D_{\gamma}:L^2(\Gamma)\to L^2(\Gamma),\,\,\,\,\, f(\gamma')\mapsto f(\gamma'-\gamma).\]
Also, denote by $\delta_\gamma\in \hat{\hat{\Gamma}}$, the scalar function $\delta_\gamma (\xi)=\xi(\gamma)$ for each $\xi\in \hat{\Gamma}$.
Note that the map $\delta:\Gamma\to\hat{\hat{\Gamma}}$, $\gamma\mapsto \delta_\gamma$, is the Pontryagin map (\cite{fol}).

\begin{proposition}
\label{p:multiplication}
Let $\gamma\in \Gamma$ and let $M_{\delta_\gamma}$ be the multiplication operator on $L^2(\hat{\Gamma})$ by the scalar function $\delta_\gamma$. Then,
\[M_{\delta_\gamma}^\ast=F D_{\gamma} F^{-1}.\] 
\end{proposition}

\begin{proof}
Let $f\in (L^1\cap L^2) (\Gamma)$ such that $\hat{f}\in L^1(\hat{\Gamma})$. For every $\xi\in \hat{\Gamma}$ we have
\begin{eqnarray*}
(F D_{\gamma} F^{-1} \hat{f})(\xi)&=&\int_\Gamma (D_{\gamma} F^{-1} \hat{f})(x) \overline{\xi(x)}dx\\
&=& \int_\Gamma (F^{-1} \hat{f})(x-\gamma) \overline{\xi(x)}dx\\
&\stackrel{x-\gamma \rightarrow x}{=}& \int_\Gamma (F^{-1} \hat{f})(x) \overline{\xi(x+\gamma)}dx\\
&=& \int_\Gamma (F^{-1} \hat{f})(x) \overline{\xi(x)}dx \, \overline{\xi(\gamma)}\\
&=& (F F^{-1} \hat{f})(\xi) \overline{\xi(\gamma)}\\
&=& \overline{\xi(\gamma)} \hat{f}(\xi)\\
&=& \overline{\delta_\gamma(\xi)} \hat{f}(\xi).
\end{eqnarray*}
Thus, if follows that $F D_{\gamma} F^{-1} \hat{f}= \overline{\delta_\gamma}\hat{f}$. The result now follows since the set of such functions $\hat{f}$ forms a dense subspace in $L^2(\hat{\Gamma})$ (\cite{osb,rud}).
\end{proof}

\begin{corollary}
\label{c:multiplication}
Let $L\in B(L^2(\Gamma))$ and define $\Lambda= F L F^{-1}\in B(L^2(\hat{\Gamma}))$. 
Then, for each $\gamma\in \Gamma$, the following statements are equivalent.
\begin{enumerate}[(i)]
\item
$D_{\gamma} L=L D_{\gamma}$.
\item
$M_{\delta_\gamma}^\ast \Lambda =\Lambda M_{\delta_\gamma}^\ast$.
\end{enumerate}
\end{corollary}

\proof
Let $\gamma\in \Gamma$.
Note that $D_{\gamma} L=L D_{\gamma}$ if and only if
\[F D_{\gamma} F^{-1} F L F^{-1}= F L F^{-1}F D_{\gamma} F^{-1}.\]
The result now follows by Proposition \ref{p:multiplication}.
\endproof

\begin{proposition}
\label{t:intertwining}
Let $L\in B(L^2(\Gamma))$. Then $L$ satisfies the commuting property
$D_{\gamma} L=L D_{\gamma}$ for all $\gamma\in \Gamma$ if and only if $L$ is unitarily equivalent to a multiplication operator $M_\Phi\in B(L^2(\hat{\Gamma}))$ for some $\Phi\in L^\infty(\hat{\Gamma})$. In particular, $L =F^{-1}M_\Phi F$.
\end{proposition}

\begin{proof}
Suppose first that $L\in B(L^2(\Gamma))$ and 
$D_{\gamma} L=L D_{\gamma}$ for all $\gamma\in \Gamma$.
By Corollary \ref{c:multiplication},  setting $\Lambda= F L F^{-1}\in B(L^2(\hat{\Gamma}))$, we obtain that 
\[ M_{\delta_\gamma}^\ast \Lambda =\Lambda M_{\delta_\gamma}^\ast,\]
for all $\gamma\in \Gamma$.
Let $f,g\in L^2(\hat{\Gamma}) \cap L^\infty(\hat{\Gamma})$. 
Then, for all $\gamma\in \Gamma$,
\begin{eqnarray*}
F((\Lambda  f)\overline{g}) (\gamma) 
&=& \int_{\hat{\Gamma}} \overline{\delta_\gamma(\xi)} (\Lambda  f)(\xi)\overline{g(\xi)} d\xi \\
&=& \int_{\hat{\Gamma}} (M_{\delta_\gamma}^\ast \Lambda f)(\xi)\overline{g(\xi)} d\xi \\
&=& \int_{\hat{\Gamma}}  (\Lambda M_{\delta_\gamma}^\ast f)(\xi)\overline{g(\xi)} d\xi \\
&=& \langle \Lambda M_{\delta_\gamma}^\ast f, g\rangle_{L^2(\hat{\Gamma})} 
\end{eqnarray*}
Similarly, for all $\gamma\in \Gamma$,
\begin{eqnarray*}
F(f(\overline{\Lambda^* g}))(\gamma) &=& 
\int_{\hat{\Gamma}}  \overline{\delta_\gamma(\xi)} f(\xi)\overline{\Lambda^* g(\xi)} d\xi \\
&=& \int_{\hat{\Gamma}}  (M_{\delta_\gamma}^\ast f)(\xi)\overline{\Lambda^* g(\xi)} d\xi \\
&=& 
\langle  M_{\delta_\gamma}^\ast f, \Lambda^* g\rangle_{L^2(\hat{\Gamma})} \\
&=& 
\langle  \Lambda M_{\delta_\gamma}^\ast f,   g\rangle_{L^2(\hat{\Gamma})} 
\end{eqnarray*}
Therefore, by the uniqueness of the Fourier transform we obtain
\[(\Lambda  f)\overline{g}=f \overline{\Lambda^* g}.\]
It now follows that, for all $h\in L^\infty(\hat{\Gamma})$,
\[\langle M_h\Lambda f,g\rangle_{L^2(\hat{\Gamma})}  = \langle M_h f, \Lambda^*g\rangle_{L^2(\hat{\Gamma})}= \langle \Lambda M_h f, g\rangle_{L^2(\hat{\Gamma})} \]
for every $f,g\in L^2(\hat{\Gamma}) \cap L^\infty(\hat{\Gamma}) $, 
and since these functions are dense in $L^2$, we get
 $M_{h}\Lambda=\Lambda M_{h}$,  so $\Lambda$ commutes with the algebra $\M_\mu$ of multiplication  operators. Thus, the result follows from Proposition \ref{p:masa}.

The reverse direction is obtained from Corollary \ref{c:multiplication}, so the proof is complete.
\end{proof}

\begin{remark}
If $\Gamma$ is a discrete abelian group and $\Phi\in L^1(\hat{\Gamma})$ then the operator $L$ in Proposition \ref{t:intertwining} satisfies, \[L(f)(\gamma') = \int_{\Gamma} \hat{\Phi}(\gamma'-\gamma)f(\gamma) d\gamma,\]
for all $\gamma'\in \Gamma$.
In particular, if $\Gamma=\bZ$ then the matrix for $L$ is the Laurent matrix with symbol $\Phi$.
\end{remark}

\subsection{Vector-valued functions}
Let $\Gamma$ be a locally compact abelian group and let $X$ and $Y$ be complex Hilbert spaces. Let also $\{x_1,x_2,\dots\}$ and $\{y_1,y_2,\dots\}$  be orthonormal bases on $X$ and $Y$, respectively. Denote by $L^2(\Gamma,X)$ the Hilbert space of square integrable $X$-valued functions. i.e. Bochner-measurable functions $f:\Gamma\to X$ such that,
\[\int_\Gamma \|f(\gamma)\|^2 \,d\gamma<\infty\] where we use normalised Haar measure on $\Gamma$.
Note that we identify the Hilbert spaces $L^2(\Gamma,X)$ and $L^2(\Gamma)\otimes X$; given any $g\in L^2(\Gamma)$, the function $g_k\in L^2(\Gamma,X)$ defined by $g_k(\gamma)=g(\gamma)x_k$, is identified with the elementary tensor $g\otimes x_k \in L^2(\Gamma)\otimes X$. 

The Fourier transform $F_X\in B(L^2(\Gamma, X),L^2(\hat{\Gamma}, X))$ is the unitary operator given by    
$F_X=F\otimes 1_X$, where $1_X$ is the identity operator on $X$. 
For each $\gamma\in \Gamma$, denote by $U_{\gamma}$ and $W_{\gamma}$ the unitary operators
\[U_{\gamma}=D_{\gamma}\otimes 1_X:L^2(\Gamma,X)\to L^2(\Gamma,X),\,\,\,\,\, f(\gamma')\mapsto f(\gamma'-\gamma),\]
\[W_{\gamma}=D_{\gamma}\otimes 1_Y:L^2(\Gamma,Y)\to L^2(\Gamma,Y),\,\,\,\,\, g(\gamma')\mapsto g(\gamma'-\gamma).\]

Given now an operator $T\in B(L^2(\Gamma, X),L^2(\Gamma, Y))$, for each 
$i,j$
let $T_{ij}\in B(L^2(\Gamma))$ be the bounded operator that is uniquely defined by the sesquilinear form,
\begin{equation}
\label{eqm1}
\langle T_{ij} f,g\rangle= \langle T(f\otimes x_j), g\otimes y_i\rangle, \, f,g\in L^2(\Gamma).
\end{equation}
We call $T_{ij}$ a {\em matrix element} of $T$.
A bounded operator $T\in B(L^2(\Gamma, X),L^2(\Gamma, Y))$ is called a \emph{multiplication operator} if
there exists $\Phi \in L^\infty(\Gamma,B(X,Y))$ such that
\[
\forall_{  f \in L^2(\Gamma, X)} \quad (Tf)(\gamma) = \Phi(\gamma) f(\gamma) \ {\rm a.e.} \ \gamma .
\]
We refer to the function $\Phi$ 
as the {\em operator-valued symbol function} for $T$ and we write $T=M_{\Phi}$.
In terms of the matrix elements $T_{ij}$ from \eqref{eqm1}, we have
$T_{ij}=M_{\Phi_{ij}}$ where $\Phi_{ij}\in L^\infty(\Gamma)$.

\begin{proposition}
\label{t:intertwiningX}
Let $L\in B(L^2(\Gamma,X),L^2(\Gamma,Y))$. Then $L$ satisfies the intertwining property
$W_{\gamma} L=L U_{\gamma}$ for all $\gamma\in \Gamma$ if and only if $L$ is unitarily equivalent to a multiplication operator $M_{\Phi}
 \in B(L^2(\hat{\Gamma},X),L^2(\hat{\Gamma},Y))$ for some $\Phi\in L^\infty(\hat{\Gamma}, 
B(X,Y))$.
In particular, $L =F_Y^{-1}M_\Phi F_X$.
\end{proposition}

\begin{proof}
Suppose that the intertwining property holds. Then for every $f,g\in L^2(\Gamma)$  we have
\[\langle  L (f\otimes x_j), W_{\gamma}^\ast (g\otimes y_i)\rangle
= \langle W_{\gamma} L (f\otimes x_j),g\otimes y_i\rangle 
= \langle LW_{\gamma} (f\otimes x_j),g\otimes y_i\rangle.\]
Equivalently, by the definition of $W_{\gamma}$,
\[\langle  L (f\otimes x_j), (D_{\gamma}^\ast g)\otimes y_i\rangle 
= \langle L ((D_{\gamma}f)\otimes x_j),g\otimes y_i\rangle.\]
This implies, 
\[
\langle  L_{ij} f, (D_{\gamma}^\ast g)\rangle 
=\langle L_{ij} (D_{\gamma}f),g\rangle,\]
which implies
\[\langle  D_{\gamma}L_{ij} f, g\rangle 
=\langle L_{ij} D_{\gamma}f,g\rangle.
\]
Thus, for each $i,j$, the operator $L_{ij}$ commutes with $D_{\gamma}$, for all $\gamma\in \Gamma$. Hence by Proposition \ref{t:intertwining}, for each $i,j$ we have $L_{ij}=F_Y^{-1} M_{\Phi_{ij}}F_X$, for some $\Phi_{ij}\in L^\infty (\hat{\Gamma})$.

Define $T = F_Y L F_X^{-1}$. This is a bounded operator that satisfies
\[
(F_Y U_{\gamma} F^{-1}_Y )T = T (F_X W_{\gamma} F^{-1}_X ) \quad \forall \gamma \in \Gamma.
\]
As $T_{ij} = M_{\Phi_{ij}}$, we conclude that $T = M_\Phi$, where $\Phi$ is the $B(X,Y)$ valued function
with matrix elements $\Phi_{i,j}$. Moreover
\[
\| \Phi \|_{L^\infty (\hat{\Gamma}, B(X,Y)) } = \| T \| = \| L \| .
\]

Once again, the reverse direction follows by straightforward calculations.
\end{proof}

\subsection{Intertwining with a twist}
Let $U(X)$ denote the unitary group of $X$ and let $\pi:\Gamma\to U(X)$ be a unitary representation of $\Gamma$ on $X$. Define $T_{\pi}\in B(L^2(\Gamma,X))$ by $(T_\pi f)(\gamma)=\pi(-\gamma)f(\gamma)$.
For each $\gamma\in \Gamma$, define 
$U_{\gamma,\pi}\in B(L^2(\Gamma,X))$ by
$(U_{\gamma,\pi}f)(\gamma')=\pi(\gamma)f(\gamma'-\gamma)$.

\begin{lemma}
\label{l:intertwiningtwist}
Let $\pi:\Gamma\to U(X)$ be a unitary representation.
Then, for each $\gamma\in \Gamma$, 
\[T_\pi U_{\gamma,\pi}= U_{\gamma} T_\pi.\]
\end{lemma}

\proof
Given $f\in L^2(\Gamma,X)$ and $\gamma\in \Gamma$, we have
\[(T_\pi U_{\gamma,\pi}f)(\gamma' )=\pi(-\gamma' )(U_{\gamma,\pi} f)(\gamma' )=
\pi(-\gamma' )\pi(\gamma)f(\gamma' -\gamma) =\pi(\gamma-\gamma' ) f(\gamma' -\gamma),\]
while
\[(U_{\gamma} T_\pi f)(\gamma' )=(T_\pi f)(\gamma' -\gamma) =
\pi(\gamma-\gamma' ) f(\gamma' -\gamma),\]
so the proof is complete.
\endproof

\begin{theorem}
\label{t:intertwiningtwist}
Let $C\in B(L^2(\Gamma,X),L^2(\Gamma,Y))$ and let $\pi:\Gamma\to U(X)$ be a unitary representation.
Then $W_{\gamma}C=C U_{\gamma,\pi}$ for all $\gamma\in \Gamma$ if and only if $C=L T_\pi$, where $L$ is unitarily equivalent to a multiplication operator $M_\Phi\in B(L^2(\hat{\Gamma},X),L^2(\hat{\Gamma},Y))$ for some $\Phi\in L^\infty(\hat{\Gamma},
B(X,Y))$. In particular, $L =F_Y^{-1}M_\Phi F_{X}$.
\end{theorem}

\begin{proof}
Suppose $W_{\gamma}C=C U_{\gamma,\pi}$ for all $\gamma\in \Gamma$.
Then, by Lemma \ref{l:intertwiningtwist}, 
\[W_{\gamma}C T_\pi^{-1}=C U_{\gamma,\pi} T_\pi^{-1} = C T_\pi^{-1} U_\gamma\] 
for all $\gamma\in \Gamma$. 
The conclusion now follows from Proposition \ref{t:intertwiningX} on taking $L=C T_\pi^{-1}$.
Conversely, suppose $C=L T_\pi$, where $L$ is unitarily equivalent to a  multiplication operator $M_{\Phi}
 \in B(L^2(\hat{\Gamma},X),L^2(\hat{\Gamma},Y))$ for some $\Phi\in L^\infty(\hat{\Gamma}, B(X,Y))$.
By  Proposition \ref{t:intertwiningX} and Lemma \ref{l:intertwiningtwist},
for each $\gamma\in \Gamma$,
\[W_{\gamma} C = W_{\gamma} L  T_\pi = L  U_{\gamma} T_\pi
= L  T_\pi U_{\gamma, \pi} = C U_{\gamma, \pi}.\]
\end{proof}


\section{Symbol functions for symmetric frameworks}
\label{s:symbolfunctions}
In this section we introduce frameworks $(G,\varphi)$ and their associated coboundary matrices $C(G,\varphi)$. We show that the action of a discrete abelian group on $(G,\varphi)$ gives rise to a  Hilbert space coboundary operator which satisfies twisted intertwining relations of the form considered in Section \ref{s:intertwining}. 
In particular, this coboundary operator can be expressed as a composition $L T_\pi$  in the manner of Theorem \ref{t:intertwiningtwist}, where $L$ is unitarily equivalent to a multiplication operator $M_{\Phi}$. We then present an explicit formula for the operator-valued symbol function $\Phi$.

\subsection{Frameworks}
Let $X$ and $Y$ be finite dimensional complex Hilbert spaces.
A  {\em framework}  for $X$ and $Y$ is a pair $(G,\varphi)$ consisting of a simple  undirected graph $G=(V,E)$ and a collection 
$\varphi=(\varphi_{v,w})_{v,w\in V}$ of linear maps $\varphi_{v,w}:X\to Y$ with the property that $\varphi_{v,w}=0$ if $vw\notin E$ and 
$\varphi_{v,w}=-\varphi_{w,v}$ for all $vw\in E$. 
We will assume throughout this section that the vertex set $V$ is a finite or countably infinite set. The graph $G$ is said to have {\em bounded degree} if 
$\sup_{v\in V} \deg(v)<\infty$, where $\deg(v)$ denotes the degree of the vertex $v\in V$.

A {\em coboundary matrix} for $(G,\varphi)$ is a matrix $C(G,\varphi)$ with rows indexed by $E$ and columns indexed by 
$V$. The row entries for a given edge $vw\in E$ are as follows,
\begin{eqnarray*}\bordermatrix{ & &   & v & &  & & w & & \cr
vw & \cdots & 0 & \varphi_{v,w} & 0& \cdots & 0& \varphi_{w,v} & 0 & \cdots  
\cr}.\end{eqnarray*}

\begin{example}
\label{ex:framework}
Let $(G,\varphi)$ be a framework for $X$ and $Y$ where 
$G=(V,E)$ is the 4-cycle with vertex set $V=\{v_1,v_2,v_3,v_4\}$ and edge set $E=\{v_1v_2,v_2v_3,v_3v_4,v_4v_1\}$. 
A coboundary matrix for $(G,\varphi)$ has the following form (up to permutations of rows and columns),

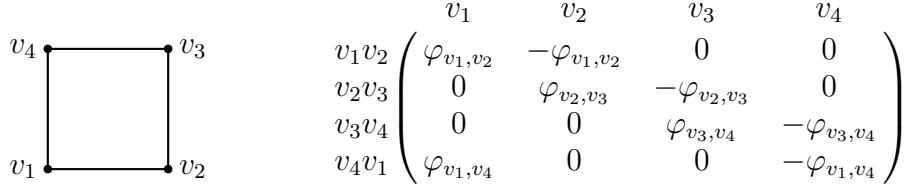
\begin{figure}[h!]
\centering
  \begin{tabular}{  c   }
  
\begin{minipage}{.3\textwidth}
\begin{tikzpicture}
\clip (-2.5,-0.5) rectangle (2.7cm,2.5cm);

\draw [thick](0.2,0.1)--(1.8,0.1);
\draw [thick](0.2,1.7)--(1.8,1.7);
\draw [fill] (0.2,0.1) circle [radius=0.05];
\draw [fill] (1.8,0.1) circle [radius=0.05];
\draw [fill] (0.2,1.7) circle [radius=0.05];
\draw [fill] (1.8,1.7) circle [radius=0.05];
\draw [thick](0.2,0.1)--(0.2,1.7);
\draw [thick](1.8,0.1)--(1.8,1.7);
\node [left] at (0.2,0.1) {$v_1$};
\node [right] at (1.8,0.1) {$v_2$};
\node [left] at (0.2,1.7) {$v_4$};
\node [right] at (1.8,1.7) {$v_3$};
\end{tikzpicture}
\end{minipage}
  
\begin{minipage}{.7\textwidth}
\begin{eqnarray*}\bordermatrix{ 
& v_1 & v_2 & v_3 & v_4 \cr
v_1v_2  & \varphi_{v_1,v_2} & -\varphi_{v_1,v_2} & 0 & 0 \cr
v_2v_3 & 0 & \varphi_{v_2,v_3} & -\varphi_{v_2,v_3} & 0 \cr
v_3v_4 & 0 & 0 & \varphi_{v_3,v_4} & -\varphi_{v_3,v_4} \cr
v_4v_1 & \varphi_{v_1,v_4} & 0 & 0 & -\varphi_{v_1,v_4} \cr}\end{eqnarray*}
\vspace{4mm}
\end{minipage}
\end{tabular}
\caption{A 4-cycle (left) and coboundary matrix (right).}\label{fig:framework}
\end{figure}
\end{example}

Note that a coboundary matrix gives rise to the linear map,
\[C(G,\varphi):X^V\to Y^E,\quad (x_v)_{v\in V}
\mapsto \left(\varphi_{v,w}(x_v-x_w)\right)_{vw\in E}.\]
We recall the following result.

\begin{proposition}{\cite[Corollary 2.9]{kas-kit-pow}}.
\label{p:bounded}
Let $(G,\varphi)$ be a framework for $X$ and $Y$.
If $G$ is a countably infinite graph with bounded degree then the following statements are equivalent.
\begin{enumerate}[(i)]
\item $\sup_{vw\in E} \,\|\varphi_{v,w}\|_{op}<\infty$.
\item $C(G,\varphi)\in B(\ell^p(V,X),\ell^p(E,Y))$, for all $p\in [1,\infty]$.
\item $C(G,\varphi)\in B(\ell^p(V,X),\ell^p(E,Y))$, for some $p\in [1,\infty]$.
\end{enumerate}
\end{proposition}

\subsection{Gain graphs}
Let $\Gamma$ be an additive group with identity element $0$.
A {\em $\Gamma$-symmetric graph} is a pair $(G,\theta)$ where $G=(V,E)$ is a simple undirected graph with automorphism group $\Aut(G)$ and $\theta:\Gamma\to\textrm{Aut}(G)$ is a group homomorphism.
For convenience, we suppress $\theta$ and write $\gamma v$ instead of $\theta(\gamma)v$ for each group element $\gamma\in\Gamma$ and each vertex $v\in V$. We also write $\gamma e$ instead of $(\gamma v)(\gamma w)$ for each $\gamma\in\Gamma$ and each edge $e=vw\in E$.
The {\em orbit} of a vertex $v\in V$ (respectively, an edge $e\in E$) under $\theta$ is the set 
$[v]=\{\gamma v:\gamma\in \Gamma\}$ (respectively, $[e]=\{\gamma e:\gamma\in \Gamma\}$). 
We denote by $V_0$ the set of all vertex orbits and by $E_0$ the set of all edge orbits. 

We will assume throughout   that $\theta$ acts {\em freely} on the vertices and edges of $G$. This means  $\gamma v \neq v$ and $\gamma e \neq e$ for all $\gamma\in\Gamma \backslash \{0\}$ and for all vertices $v\in V$ and edges $e\in E$.
We will also assume that $V_0$ and $E_0$ are finite sets.

\begin{lemma}
\label{l:degree}
Let $(G,\theta)$ be a $\Gamma$-symmetric graph where $\theta$ acts freely on the vertices and edges of $G$ and $E_0$ is finite.
Then $G$ has bounded degree.
\end{lemma}

\proof
Let $v\in V$ and suppose $vw_1,vw_2,vw_3\in E$ are distinct edges which belong to the same edge orbit.
Then $vw_2=\gamma(vw_1)$ for some $\gamma\in\Gamma\backslash\{0\}$.
Since $\theta$ acts freely on $V$ it follows that $w_2=\gamma v$.
Note that $vw_3=\gamma'(vw_2)$ for some $\gamma'\in\Gamma\backslash\{0\}$.
Again, since $\theta$ acts freely on $V$ it follows that $v=\gamma'w_2=(\gamma'\gamma) v$.
Thus $\gamma'=-\gamma$ and so $vw_1=-\gamma(vw_2)=\gamma'(vw_2)=vw_3$, a contradiction. We conclude that each edge orbit contains at most two edges which are incident with $v$.
Thus $v$ has at most $2|E_0|$ incident edges.
\endproof

The {\em quotient graph} $G_0$ is the multigraph with vertex set $V_0$, edge set $E_0$ and incidence relation satisfying $[e] = [v][w]$ if some (equivalently, every) edge in $[e]$ is incident with a vertex in $[v]$ and a vertex in $[w]$. For each vertex orbit $[v]\in V_0$, choose a representative vertex $\tilde{v}\in[v]$ and denote the set of all such representatives by $\tilde{V}_0$.
Now fix an orientation on the edges of the quotient graph $G_0$ so that each edge in $G_0$ is an ordered pair $[e]=([v],[w])$. Then for each directed edge $[e]=([v],[w])$ there exists a unique group element $\gamma\in\Gamma$ such that $\tilde{v}(\gamma\tilde{w})\in [e]$. This group element is referred to as the {\em gain} on the directed edge $[e]$ and is denoted $\psi_{[e]}$.
A {\em gain graph} for the $\Gamma$-symmetric graph $(G,\theta)$ is any edge-labelled directed multigraph obtained from the quotient graph $G_0$ in this way.

\begin{example}
\label{ex:symgraph}
Consider again the 4-cycle $G=(V,E)$ with vertex set $V=\{v_1,v_2,v_3,v_4\}$ and edge set $E=\{v_1v_2,v_2v_3,v_3v_4,v_4v_1\}$. Let 
$\theta:\bZ_2\to \Aut(G)$ be the group homomorphism with $\theta(1)v_1=v_3$ and $\theta(1)v_2=v_4$.
The $\bZ_2$-symmetric graph $(G,\theta)$ has two distinct vertex orbits
$[v_1]=\{v_1,v_3\}$ and $[v_2]=\{v_2,v_4\}$, and two distinct edge orbits $[v_1v_2]=\{v_1v_2,v_3v_4\}$ and $[v_1v_3]=\{v_1v_3,v_2v_4\}$.
A gain graph for $(G,\theta)$ is illustrated in Figure \ref{fig:symgraph}. 
\end{example}

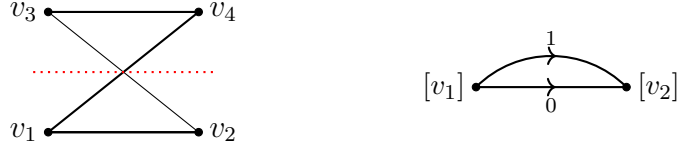
\begin{figure}[h!]
\centering
  \begin{tabular}{  c   }
  
\begin{minipage}{.5\textwidth}
\begin{tikzpicture}
\clip (-4.5,-0.4) rectangle (2.7cm,2.5cm);

\draw [thick](0,0.2)--(2,0.2);
\draw [thick](0,1.8)--(2,1.8);
\draw [fill] (0,1.8)--(2,0.2);
\draw [fill] (0,0.2) circle [radius=0.05];
\draw [fill] (2,0.2) circle [radius=0.05];
\draw [fill] (0,1.8) circle [radius=0.05];
\draw [fill] (2,1.8) circle [radius=0.05];
\draw [thick](0,0.2)--(2,1.8);
\draw [thick](0,0.2)--(2,0.2);
\node [left] at (0,0.2) {$v_1$};
\node [right] at (2,0.2) {$v_2$};
\node [left] at (0,1.8) {$v_3$};
\node [right] at (2,1.8) {$v_4$};

\draw [thick,dotted,red](-0.2,1)--(2.2,1);
\end{tikzpicture}
\end{minipage}

\begin{minipage}{.5\textwidth}
\begin{tikzpicture}
\clip (-2,-0.4) rectangle (2.7cm,2.5cm);

\draw [thick,->-] (0,0.8)--(2,0.8) node[midway,below] {\tiny{$0$}};;
\draw [thick,->-] (0,0.8)to [out=45,in=135] node[above] {\tiny{$1$}}(2,0.8);
\draw [fill] (0,0.8) circle [radius=0.05];
\draw [fill] (2,0.8) circle [radius=0.05];
\node [left] at (0,0.8) {\small{$[v_1]$}};
\node [right] at (2,0.8) {\small{$[v_2]$}};
\end{tikzpicture}
\end{minipage}

\end{tabular}
\caption{A $\bZ_2$-symmetric graph (left) and gain graph (right).}\label{fig:symgraph}
\end{figure}

For each directed edge $[e]=([v],[w])$ in the gain graph with gain $\gamma$  we choose $\tilde{e} = \tilde{v}(\gamma\tilde{w})\in E$ to be the representative edge for the edge orbit $[e]$. The set of all such representative edges will be denoted $\tilde{E}_0$.
Note that since $\theta$ acts freely on the vertex set $V$ and edge set $E$ we have natural bijections, 
\[\beta_V:\Gamma\times V_0\to V, \quad (\gamma,[v])\mapsto \gamma\tilde{v},\quad\mbox{ and, } \quad
\beta_E:\Gamma\times E_0\to E, \quad (\gamma,[e])\mapsto \gamma\tilde{e}.\]
For more on gain graphs we refer the reader to \cite{jkt}.

\subsection{Symmetric frameworks}
Let $\Gamma$ be a discrete abelian group and denote by $\Isom(X)$ the group of affine isometries of $X$.
A {\em $\Gamma$-symmetric   framework} is a tuple $\G=(G,\varphi,\theta,\tau)$ where $\tau:\Gamma\rightarrow \Isom(X)$ is a group homomorphism,
$(G,\theta)$ is a $\Gamma$-symmetric graph and $(G,\varphi)$ is a framework for $X$ and $Y$ with the property that, 
\[\varphi_{\gamma v,\gamma w} = \varphi_{v,w}\circ\tau(-\gamma),\quad \mbox{ for all } \gamma\in \Gamma \mbox{ and all } v,w\in V.\]

For each $\gamma\in\Gamma$, let $d\tau(\gamma)$ denote the linear isometry on $X$ that is uniquely defined by the linear part of the affine isometry $\tau(\gamma)$.
We denote by $\tilde{\tau}:\Gamma\to U(X^{V_0})$ the unitary representation with $\tilde{\tau}(\gamma)(x) = (d\tau(\gamma)x_{[v]})_{[v]\in V_0}$ for all $x=(x_{[v]})_{[v]\in V_0}\in X^{V_0}$.

Given a vector $z=(z_v)_{v\in V}\in X^V$ we will write 
$z_e=z_v-z_w$ for each edge $e=vw\in E$ where the corresponding directed edge 
$[e]$ in the gain graph is directed from $[v]$ to $[w]$. We will also write $\varphi_e=\varphi_{v,w}$ for such an edge.

For each $p\in[1,\infty]$, the bijections $\beta_V$ and $\beta_E$ give rise to isometric isomorphisms,
\[S_V:  \ell^p(V,X)\to \ell^p(\Gamma,X^{V_0}), \quad z=(z_v)_{v\in V}\mapsto (S_V(z)_\gamma)_{\gamma\in \Gamma},\]
where $S_V(z)_\gamma = (z_{\gamma\tilde{v}})_{[v]\in V_0}$, and,
\[S_E: \ell^p(E,Y) \to \ell^p(\Gamma,Y^{E_0}), \quad z=(z_e)_{e\in E}\mapsto (S_E(z)_\gamma)_{\gamma\in \Gamma},\]
where $S_E(z)_\gamma = (z_{\gamma\tilde{e}})_{[e]\in E_0}$.
We define the bounded operator, 
\[\tilde{C}(G,\varphi):=S_E\circ C(G,\varphi)\circ S_V^{-1}:\ell^p(\Gamma,X^{V_0})\to \ell^p(\Gamma,Y^{E_0}).\] 

For each $p\in[1,\infty]$ and each $\gamma\in \Gamma$,
we have an associated pair of isometric isomorphisms $U_{\gamma,\tilde{\tau}}\in B(\ell^p(\Gamma,X^{V_0}))$ and 
$W_{\gamma}\in B(\ell^p(\Gamma,Y^{E_0}))$ where,
\[(U_{\gamma,\tilde{\tau}}f)(\gamma')=\tilde{\tau}(\gamma)f(\gamma'-\gamma),\quad \forall\,f\in \ell^p(\Gamma,X^{V_0}),\]
\[(W_{\gamma}g)(\gamma')=g(\gamma'-\gamma),\quad \forall\,g\in \ell^p(\Gamma,Y^{E_0}).\]
 
\begin{proposition}
\label{p:rigidityoperator}
Let $\G=(G,\varphi,\theta,\tau)$ be a $\Gamma$-symmetric framework for $X$ and $Y$. Then, for all $\gamma\in \Gamma$,
 \[W_{\gamma}\circ \tilde{C}(G,\varphi)=\tilde{C}(G,\varphi)\circ U_{\gamma,\tilde{\tau}}.\] 
\end{proposition}

\begin{proof}
Let $\gamma\in \Gamma$ and let $f\in \ell^p(\Gamma, X^{V_0})$. 
Then $f=S_V(u)$ where $u=(u_v)_{v\in V}\in\ell^p(V,X)$ has components
$u_v=f(\gamma')_{[v]}$ for  $v=\beta_V(\gamma',[v])$.
We have,
\[\tilde{C}(G,\varphi)(f) = S_E\circ C(G,\varphi)\circ S_V^{-1} (f) 
= S_E\left(\varphi_{v,w}(u_v-u_w)\right)_{vw\in E}
= g,\]
where $g\in \ell^p(\Gamma, Y^{E_0})$ satisfies 
$g(\gamma') =( \varphi_{\gamma'\tilde{e}}(u_{\gamma' \tilde{e}}) )_{[e]\in E_0}$ for each $\gamma'\in \Gamma$.
Note that,
\[W_{\gamma}(g)(\gamma') = 
( \varphi_{(\gamma'-\gamma)\tilde{e}}(u_{(\gamma'-\gamma) \tilde{e}}) )_{[e]\in E_0}, \quad \mbox{ for each }\gamma'\in \Gamma.\]

Let $h = U_{\gamma,\tilde{\tau}} (f)$.
Then $h\in \ell^p(\Gamma, X^{V_0})$ and 
$h(\gamma') = \tilde{\tau}(\gamma)f(\gamma'-\gamma)$ for each $\gamma'\in \Gamma$.
Also, if $v=\beta_V(\gamma',[v])$ then,
\[h(\gamma')_{[v]} = d\tau(\gamma)f(\gamma'-\gamma)_{[v]} 
= d\tau(\gamma)u_{(\gamma'-\gamma)\tilde{v}} 
= d\tau(\gamma)u_{-\gamma v}.\]
Thus $h=S_V(z)$ where $z=(z_v)_{v\in V}\in\ell^p(V,X)$ has components
$z_v = d\tau(\gamma)u_{-\gamma v}$ for all $v\in V$.
We conclude that,
\[(\tilde{C}(G,\varphi)\circ U_{\gamma,\tilde{\tau}}) f 
= S_E\circ C(G,\varphi)\circ S_V^{-1}(h) 
= S_E\left(\varphi_e(z_e)\right)_{e\in E}
= \tilde{g},\]
where $\tilde{g}\in \ell^p(\Gamma, Y^{E_0})$ satisfies
$\tilde{g}(\gamma') =( \varphi_{\gamma'\tilde{e}}(z_{\gamma' \tilde{e}}) )_{[e]\in E_0}$ for each $\gamma'\in \Gamma$.
It remains to show that $W_{\gamma}(g)=\tilde{g}$. To see this, note that for each $[e]\in E_0$ and each $\gamma'\in \Gamma$ we have,
\[ \varphi_{\gamma'\tilde{e}}(z_{\gamma' \tilde{e}})
= \varphi_{\gamma'\tilde{e}}(d\tau(\gamma)u_{(\gamma'-\gamma)\tilde{e}}) 
= \varphi_{\gamma'\tilde{e}}(\tau(\gamma)u_{(\gamma'-\gamma)\tilde{e}}) \\
= \varphi_{(\gamma'-\gamma)\tilde{e}}(u_{(\gamma'-\gamma)\tilde{e}}).
\]
\endproof

For each $p\in [1,\infty]$, the unitary representation $\tilde{\tau}:\Gamma\to U(X^{V_0})$ defined above gives rise to an isometric isomorphism  
$T_{\tilde{\tau}}\in B(\ell^p(\Gamma,X^{V_0}))$ where, 
\[(T_{\tilde{\tau}} f)(\gamma)=\tilde{\tau}(-\gamma)f(\gamma),\quad \forall\,f\in \ell^p(\Gamma,X^{V_0}).\]

\begin{theorem}
\label{t:rigidityoperator}
Let $\G=(G,\varphi,\theta,\tau)$ be a $\Gamma$-symmetric framework for  $X$ and $Y$ where $G$ has a finite or a countably infinite vertex set, $\Gamma$ is a discrete abelian group, $\theta$ acts freely on the vertices and edges of $G$ and $V_0$ and $E_0$ are finite sets. 

Then $C(G,\varphi)\in B(\ell^2(V,X),\ell^2(E,Y))$ and,  
\[C(G,\varphi)=S_E^{-1}\circ F_{Y^{E_0}}^{-1}\circ M_\Phi \circ F_{X^{V_0}} \circ T_{\tilde{\tau}} \circ S_V,\] 
for some $\Phi\in L^\infty(\hat{\Gamma},B(X^{V_0},Y^{E_0}))$.
\end{theorem}

\proof
By Lemma \ref{l:degree}, $G$ has bounded degree.
Note that  $\varphi$ satisfies Proposition \ref{p:bounded}$(i)$ and so $C(G,\varphi)\in B(\ell^2(V,X),\ell^2(E,Y))$.
The result now follows  from Theorem \ref{t:intertwiningtwist} and Proposition \ref{p:rigidityoperator}.
\endproof

We refer to $\Phi$ in the above theorem as the {\em symbol function} for the symmetric framework $\G$. 

\subsection{The symbol function}
Let $\G=(G,\varphi,\theta,\tau)$ be a $\Gamma$-symmetric   framework for $X$ and $Y$ where $\Gamma$ is a discrete abelian group. 
Fix a gain graph for the $\Gamma$-symmetric graph $(G,\theta)$ and let $\chi\in\hat{\Gamma}$. A {\em $\chi$-orbit matrix} for $\G$ is a matrix $O_\G(\chi)$ with rows indexed by the directed edges of the gain graph and with columns indexed by 
$V_0$. 
The row entries for a non-loop directed edge $([v],[w])\in E_0$ with gain $\gamma\in \Gamma$ are as follows,
\begin{eqnarray*}\bordermatrix{ & &   & [v] & & & & [w] & &    \cr
& \cdots & 0 & \varphi_{\tilde{v},\gamma\tilde{w}} & 0 & \cdots & 0 & \chi(\gamma)\varphi_{\tilde{w}, -\gamma\tilde{v}}  & 0 & \cdots \cr}.\end{eqnarray*}
The row entries for a loop edge $([v],[v])\in E_0$ with gain $\gamma\in \Gamma$ are as follows,
\begin{eqnarray*}\bordermatrix{ & &   & [v] & &    \cr
& \cdots & 0 & \varphi_{\tilde{v},\gamma\tilde{v}}+\chi(\gamma)\varphi_{\tilde{v},-\gamma\tilde{v}} & 0 & \cdots \cr}.\end{eqnarray*}
Note that each orbit matrix gives rise in  natural way to a linear map $O_\G(\chi):X^{V_0}\to Y^{E_0}$ and that the function $O_\G:\hat{\Gamma}\to B(X^{V_0},Y^{E_0})$, $\chi\mapsto O_\G(\chi)$, is continuous. 
In particular, $O_\G\in L^\infty(\hat{\Gamma},B(X^{V_0},Y^{E_0}))$ is the operator-valued symbol function for a multiplication operator $M_{O_\G}\in B(L^2(\hat{\Gamma},X^{V_0}),L^2(\hat{\Gamma},Y^{E_0}))$.

We now show that $O_\G$ is the symbol function for the  symmetric framework $\G$.

\begin{theorem}
\label{t:symbol}
Let $\G=(G,\varphi,\theta,\tau)$ be a $\Gamma$-symmetric framework with  symbol function $\Phi\in L^\infty(\hat{\Gamma},B(X^{V_0},Y^{E_0}))$.
Then, 
\[\Phi(\chi)=O_\G(\chi),\quad \mbox{ a.e.~ }\chi\in\hat{\Gamma}.\]
\end{theorem}

\proof
Let $\hat{f}\in L^2(\hat{\Gamma}, X^{V_0})$ and let 
$f = F_{X^{V_0}}^{-1}(\hat{f})\in \ell^2(\Gamma, X^{V_0})$. 
Note that $(T_{\tilde{\tau}}^{-1}f)(\gamma) = \tilde{\tau}(\gamma)f(\gamma)$.
Thus $T_{\tilde{\tau}}^{-1}(f)=S_V(z)$ where $z=(z_v)_{v\in V}\in\ell^2(V,X)$ has components $z_v=(\tilde{\tau}(\gamma)f(\gamma))_{[v]}$
for  $v=\beta_V(\gamma,[v])$.
Now,
\[\tilde{C}(G,\varphi)\circ T_{\tilde{\tau}}^{-1}(f) 
= S_E\circ C(G,\varphi)\circ S_V^{-1} \circ T_{\tilde{\tau}}^{-1}(f) 
= S_E\left(\varphi_e(z_e)\right)_{e\in E}
= g,\]
where $g\in \ell^2(\Gamma, Y^{E_0})$ satisfies
$g(\gamma) =( \varphi_{\gamma\tilde{e}}(z_{\gamma \tilde{e}}) )_{[e]\in E_0}$  for each $\gamma\in \Gamma$.

Let  $[e]=([v],[w])\in E_0$ be a directed edge with gain $\gamma\in \Gamma$ and let $g_{[e]}\in \ell^2(\Gamma,Y)$ be the $[e]$-component of $g$.
Note that for each $\gamma'\in\Gamma$,
\begin{eqnarray*}
g_{[e]}(\gamma')
&=& \varphi_{\gamma\tilde{e}}(z_{\gamma \tilde{e}}) \\
&=& \varphi_{\gamma'\tilde{e}}(d\tau(\gamma')f(\gamma')_{[v]}
-d\tau(\gamma'+\gamma)f(\gamma'+\gamma)_{[w]}) \\
&=& \varphi_{\tilde{e}}(f(\gamma')_{[v]}
-d\tau(\gamma)((U_{-\gamma}f)(\gamma')_{[w]})).
\end{eqnarray*}
Also, by Proposition \ref{p:multiplication}, for almost every $\chi\in\hat{\Gamma}$,
\[\widehat{U_{-\gamma}f}(\chi) 
=\overline{\delta_{-\gamma}(\chi)}\hat{f}(\chi)  
=\overline{\chi(-\gamma)}\hat{f}(\chi) 
=\chi(\gamma)\hat{f}(\chi),\]
and so,
\[\hat{g}_{[e]}(\chi) 
= \varphi_{\tilde{e}}(\hat{f}(\chi)_{[v]}
-d\tau(\gamma)(\chi(\gamma)\hat{f}(\chi)_{[w]})) \\
= \varphi_{\tilde{v},\gamma\tilde{w}}(\hat{f}(\chi)_{[v]})
+\chi(\gamma)\varphi_{\tilde{w},-\gamma\tilde{v}}(\hat{f}(\chi)_{[w]}).\]
Thus, for almost every $\chi\in\hat{\Gamma}$,
\[
(M_\Phi\hat{f})(\chi)
=(F_{Y^{E_0}}\circ \tilde{C}(G,\varphi) \circ T_{\tilde{\tau}}^{-1} f)(\chi)
=\hat{g}(\chi) = O_\G(\chi)\hat{f}(\chi).\]
\end{proof}

\begin{corollary}
\label{c:block}
Let $\G=(G,\varphi,\theta,\tau)$ be a $\Gamma$-symmetric framework
with  symbol function $\Phi$. If $G$ is a finite graph then the coboundary matrix $C(G,\varphi)$ is equivalent to the direct sum, 
\[\bigoplus_{\chi\in\hat{\Gamma}} O_\G(\chi):\bigoplus_{\chi\in\hat{\Gamma}} X^{V_0}\to \bigoplus_{\chi\in\hat{\Gamma}} Y^{E_0}.\]
\end{corollary}

\proof
By Theorem \ref{t:rigidityoperator}, $C(G,\varphi)$ is equivalent to $M_\Phi$.
Note that since $G$ is a finite graph and $\theta$ acts freely on the vertices and edges of $G$ it follows that $\Gamma$, and hence also $\hat{\Gamma}$, is  finite. Thus, $M_\Phi$ is equivalent to the direct sum $\oplus_{\chi\in\hat{\Gamma}} \Phi(\chi)$. Also, by Theorem \ref{t:symbol}, $\Phi(\chi) = O_\G(\chi)$ for all $\chi\in \hat{\Gamma}$ and so the result follows.
\endproof

\begin{example}
\label{ex:symframework}
Consider again the framework $(G,\varphi)$ in Example \ref{ex:framework} and let  $(G,\theta)$ be the $\bZ_2$-symmetric graph described in Example \ref{ex:symgraph}. 
Let $[e_1]$ be the directed edge in the accompanying gain graph with gain $0$ and let $[e_2]$ be the directed edge with gain $1$. 
Note that the dual group for $\bZ_2$ consists of characters $\chi_0$ and $\chi_1$ which satisfy $\chi_0(1)=1$ and $\chi_1(1)=-1$.
If $\G=(G,\varphi,\theta,\tau)$ is a $\bZ_2$-symmetric framework
then the associated orbit matrices for $\G$ take the following form, 
\begin{eqnarray*}
O_\G(\chi_0) =\bordermatrix{ 
& [v_1] & [v_2] \cr
\left[e_1\right] & \varphi_{\tilde{v}_1,\tilde{v}_2} & -\varphi_{\tilde{v}_1,\tilde{v}_2}\cr
\left[e_2\right] & \varphi_{\tilde{v}_1,\tilde{v}_4}  & \varphi_{\tilde{v}_2,\tilde{v}_3}\cr},
\quad\quad
O_\G(\chi_1)=\bordermatrix{ 
& [v_1] & [v_2] \cr
\left[e_1\right] & \varphi_{\tilde{v}_1,\tilde{v}_2} & -\varphi_{\tilde{v}_1,\tilde{v}_2} \cr
\left[e_2\right] & \varphi_{\tilde{v}_1,\tilde{v}_4}  & -\varphi_{\tilde{v}_2,\tilde{v}_3} \cr}.
\end{eqnarray*}
Applying Corollary \ref{c:block} we obtain the equivalence,
\[C(G,\varphi)\sim \begin{bmatrix} O_\G(\chi_0) & 0 \\
0 & O_\G(\chi_1) \end{bmatrix}.\]
\end{example}

\begin{corollary}\label{cor:FCoef}
Let $\G=(G,\varphi,\theta,\tau)$ be a $\Gamma$-symmetric framework with  symbol function $\Phi=O_\G\in C(\hat{\Gamma},B(X^{V_0},Y^{E_0}))$.
Fix a gain graph for $(G,\theta)$ and let $\Gamma_0\subset\Gamma$ be the finite set of non-zero gains on the edges of this gain graph.
\begin{enumerate}[(i)]
\item $\Phi$ is the operator-valued trigonometric polynomial with, 
\[\Phi(\chi)=\hat{\Phi}(0)+\sum\limits_{\gamma\in \Gamma_0} \hat{\Phi}(\gamma) \chi(\gamma),\quad \forall\,\chi\in\hat{\Gamma}.\] 

\item For each $\gamma\in \Gamma_0$, each $[v]\in V_0$ and each $[e]\in E_0$,
\[\hat{\Phi}(\gamma)_{[e],[v]}= C(G,\varphi)_{\tilde{e},\gamma\tilde{v}}\circ d\tau(\gamma),\]
where $\hat{\Phi}(\gamma)_{[e],[v]}$ is the $([e],[v])$-entry of $\hat{\Phi}(\gamma)$ and  $C(G,\varphi)_{\tilde{e},\gamma\tilde{v}}$ is the $(\tilde{e},\gamma\tilde{v})$-entry of $C(G,\varphi)$. 
\end{enumerate}
\end{corollary}

\begin{remark}
The orbit matrix $O_\G(1_{\hat{\Gamma}})$ was first introduced in \cite{sch-whi} in the context of finite bar-joint frameworks $(G,p)$ with an abelian symmetry group. There the linear maps $\varphi_{v,w}$ are derived from Euclidean distance constraints and the orbit matrix is used to analyze fully symmetric motions of the framework in Euclidean space $\bR^d$. The general orbit matrices $O_\G(\chi)$ were later introduced in \cite{schtan} and used to derive the block-diagonalisation result in Corollary \ref{c:block}.

The symbol function $\Phi$ for periodic bar-joint frameworks in $\bR^d$, again with Euclidean distance constraints, was first introduced in \cite{owe-pow}. In this setting the symmetry group is $\bZ^d$ and the dual group is the $d$-torus $\bT^d$. It is proved there that the rigidity matrix for the framework determines a Hilbert space operator $R(G,p):\ell^2(V,\bC^d)\to \ell^2(E,\bC)$ which is unitarily equivalent to the multiplication operator $M_\Phi:L^2(\bT^d,\,\bC^{d|V_0|})\to L^2(\bT^d,\bC^{|E_0|})$.

Theorem \ref{t:symbol} unifies and generalises these two contexts to frameworks with a general (finite or infinite) discrete abelian symmetry group and arbitrary linear edge constraints. See Section \ref{s:examples} for some examples.  
\end{remark}

\section{A generalised RUM spectrum}
\label{s:RUM}
Let $\G=(G,\varphi,\theta,\tau)$ be a $\Gamma$-symmetric framework for $X$ and $Y$ with  symbol function $\Phi\in C(\hat{\Gamma},B(X^{V_0},Y^{E_0}))$. 
Fix $\chi\in\hat{\Gamma}$ and  $a\in X^{V_0}$ and define 
$z(\chi,a)=(z_v)_{v\in V}\in\ell^\infty(V, X)$ to be the bounded vector with components,
\[z_v = \chi(\gamma)d\tau(\gamma)a_{[v]},\quad \mbox{ for }v=\beta_V(\gamma,[v]).\]
We refer to $z(\chi,a)$ as a {\em $\chi$-symmetric} vector in $\ell^\infty(V, X)$.

In this section our aim is to prove the following result.

\begin{theorem}
\label{t:twistedflex}
If  $a\in\operatorname{ker}\Phi(\chi)$ then $z(\chi,a)\in \ker C(G,\varphi)$.
\end{theorem}

\subsection{Key lemmas}
Let $(u_\lambda)_{\lambda\in \Lambda}$ be an approximate identity for $L^1(\hat{\Gamma})$ where, for each $\lambda\in\Lambda$,
 $u_\lambda$ is a positive continuous function satisfying
$u_\lambda(\eta ) = u_\lambda(\eta^{-1} )$ for all $\eta \in \hat{\Gamma}$ 
and  $\|u_\lambda\|_1=1$.
It is a standard procedure to show that,  
\[\|u_\lambda* f - f\|_p\to 0,\]
for all $p\in[1,\infty)$ when $f\in L^p(\hat{\Gamma})$ and for $p=\infty$ when $f\in C(\hat{\Gamma})$.
(See \cite[Proposition 2.42]{fol} eg.)
Note that since $u_\lambda(\eta ) = u_\lambda(\eta^{-1} )$ for all $\eta \in \hat{\Gamma}$ it follows that $\check{u}_\lambda=\hat{u}_\lambda\in C_0(\Gamma)$.

For each $\lambda\in\Lambda$,  denote by $u_{\lambda,a}:\hat{\Gamma}\to X^{V_0}$ the function $\eta \mapsto u_\lambda(\eta )a$ and define $\psi_\lambda \in C(\hat{\Gamma},Y^{E_0})^*$ by, 
\[\psi_\lambda(g) = 
\int_{\hat{\Gamma}}\, \langle\, \Phi(\eta )(u_{\lambda, a}(\chi^{-1}\eta)), \, g(\eta )\,\rangle \, d\eta ,\quad \forall\,g\in C(\hat{\Gamma},Y^{E_0}).\]

\begin{lemma}
\label{l:twistedflex}
If  $a\in\operatorname{ker}\Phi(\chi)$ then $\psi_\lambda\stackrel{w^\ast}{\rightarrow}0$. 
\end{lemma}

\proof
Let $g\in C(\hat{\Gamma},Y^{E_0})$ and define
$f\in C(\hat{\Gamma})$ by,
\[f(\eta ) = \langle \Phi(\chi \eta)a,\, g(\chi \eta)\rangle,\quad \forall\,\eta\in \hat{\Gamma}.\]
Note that $f(1_{\hat{\Gamma}}) =  \langle \Phi(\chi )a,\, g(\chi ) \rangle =0$.
We have,
\[\psi_\lambda(g)
= \int_{\hat{\Gamma}} u_\lambda(\chi ^{-1}\eta)\langle \Phi(\eta ) a,\, g(\eta )\rangle\, d\eta  
= (u_\lambda*f)(1_{\hat{\Gamma}}) 
\rightarrow f(1_{\hat{\Gamma}}) =0.\] 
Hence $\psi_\lambda\stackrel{w^\ast}{\rightarrow}0$. 
\endproof

For each $\lambda\in\Lambda$, define 
$\nu_\lambda \in \ell^1(\Gamma,Y^{E_0})^*$ by, 
\[\nu_\lambda(g) = 
\sum_{\gamma\in\Gamma}\, \langle\, \tilde{C}(G,\varphi)\circ T_{\tilde{\tau}}^{-1} \circ M_{\delta_\chi }(\check{u}_{\lambda, a})(\gamma),\, g(\gamma) \,\rangle,\quad \forall\,g\in \ell^1(\Gamma,Y^{E_0}).\]

\begin{lemma}
\label{l:twistedflex2}
If  $a\in\operatorname{ker}\Phi(\chi)$ then $\nu_\lambda\stackrel{w^\ast}{\rightarrow}0$. 
\end{lemma}

\proof
For each $\lambda\in\Lambda$, define the continuous function
$\phi_\lambda \in C(\hat{\Gamma},Y^{E_0})$ by, 
\[\phi_\lambda(\eta ) = \Phi(\eta )(u_{\lambda, a}(\chi^{-1}\eta )).\]
By Proposition \ref{p:multiplication}  and Theorem \ref{t:rigidityoperator} we obtain,
\begin{eqnarray*}
\check{\phi}_\lambda
&=& \tilde{C}(G,\varphi)\circ T_{\tilde{\tau}}^{-1}\circ M_{\delta_\chi }(\check{u}_{\lambda, a}).
\end{eqnarray*}
Let $g\in \ell^1(\Gamma, Y^{E_0})$. 
Then $\hat{g}\in C(\hat{\Gamma},Y^{E_0})$ and so, using Lemma  \ref{l:twistedflex}, we have,
\begin{eqnarray*}
\nu_\lambda(g) 
&=& \sum_{\gamma\in\Gamma}\, \langle \check{\phi}_\lambda(\gamma),\, g(\gamma)\rangle \\
&=& \sum_{\gamma\in\Gamma}\, \langle\, \int_{\hat{\Gamma}} \eta (\gamma) \phi_\lambda(\eta ) \,d\eta,\,  g(\gamma)\,\rangle \\
&=& \int_{\hat{\Gamma}}\, \langle \phi_\lambda(\eta ),\,  \sum_{\gamma\in\Gamma} \overline{\eta (\gamma)} g(\gamma) \rangle \, d\eta  \\
&=& \int_{\hat{\Gamma}} \,\langle \phi_\lambda(\eta ),\, \hat{g}(\eta )\rangle\, d\eta \\
&=& \psi_\lambda(\hat{g})
\rightarrow 0.
\end{eqnarray*}
Thus $\nu_\lambda\stackrel{w^\ast}{\rightarrow}0$. 
\endproof

Denote by $\chi \otimes a:\Gamma\to  X^{V_0}$ the function $\gamma\mapsto \chi (\gamma)a$ and define
$\rho(\chi,a)  \in \ell^1(\Gamma,Y^{E_0})^*$ by, 
\[\rho(\chi,a)(g) = 
\sum_{\gamma\in\Gamma}\, \langle\, \tilde{C}(G,\varphi)\circ T_{\tilde{\tau}}^{-1}(\chi \otimes a)(\gamma),\, g(\gamma) \,\rangle,\quad \forall\,g\in \ell^1(\Gamma,Y^{E_0}).\]

\begin{lemma}
\label{l:twistedflex3}
$\nu_\lambda\stackrel{w^\ast}{\rightarrow} \rho(\chi,a)$.
\end{lemma}

\proof
Let $g\in \ell^1(\Gamma,Y^{E_0})$ and let $\epsilon>0$.
Choose a finite subset $K\subset \Gamma$ such that $\sum\limits_{\gamma\notin K} \|g(\gamma)\|<\epsilon$.
By \cite[Lemma 4.46]{fol}, $\check{u}_\lambda \rightarrow 1$ uniformly on compact subsets of $\Gamma$ and so there exists $\lambda'\in\Lambda$ such that $\max_{\gamma\in K} |\check{u}_{\lambda}(\gamma)-1|<\epsilon$ for all $\lambda\geq \lambda'$.


Define $f_\lambda\in \ell^\infty(\Gamma,X^{V_0})$ by setting
$f_\lambda  = M_{\delta_\chi}(\check{u}_{\lambda, a})-(\chi \otimes a)$ for each $\lambda\in \Lambda$.
Since $\|\check{u}_\lambda\|_\infty \leq \|u_\lambda\|_1=1$ we have, \[\|f_\lambda\|_\infty = \sup_{\gamma\in\Gamma} \|\chi(\gamma)(\check{u}_{\lambda}(\gamma)-1)a\| \leq 2\|a\|.\]

Let $1_K$ denote the characteristic function for $K$. Then for all $\lambda\geq \lambda'$ we have,
 \[\|f_\lambda1_K\|_\infty = \max_{\gamma\in K} \|\chi(\gamma)(\check{u}_{\lambda}(\gamma)-1)a\| = \max_{\gamma\in K} |\check{u}_{\lambda}(\gamma)-1|\|a\|<\|a\|\epsilon.\]
Note that, by Proposition \ref{p:bounded}, $\tilde{C}(G,\varphi)\circ T_{\tilde{\tau}}^{-1}\in B(\ell^\infty(\Gamma,  X^{V_0}),\ell^\infty(\Gamma,  Y^{E_0}))$.
Moreover, $T_{\tilde{\tau}}^{-1}$ is isometric and so for all $\lambda\in \Lambda$,
\[\max_{\gamma\in K} \| \tilde{C}(G,\varphi)\circ T_{\tilde{\tau}}^{-1}(f_\lambda)(\gamma)\|
= \| \tilde{C}(G,\varphi)\circ T_{\tilde{\tau}}^{-1}(f_\lambda 1_K)\|_\infty 
\leq
\| \tilde{C}(G,\varphi)\|_{op}\|f_\lambda 1_K\|_\infty. 
\]
Thus, for all $\lambda\geq \lambda'$ we have,
\begin{eqnarray*}
\left|\left(\nu_\lambda- \rho(\chi,a)\right)(g)\right|
&\leq&
\sum_{\gamma\in \Gamma} \big|\langle \tilde{C}(G,\varphi)\circ T_{\tilde{\tau}}^{-1}(f_\lambda)(\gamma),\, g(\gamma) \rangle\big|\\
&\leq&
\sum_{\gamma\in \Gamma} \|\tilde{C}(G,\varphi)\circ T_{\tilde{\tau}}^{-1}(f_\lambda)(\gamma)\|\, \|g(\gamma)\|\\
&\leq&
\| \tilde{C}(G,\varphi)\|_{op}\|f_\lambda 1_K\|_\infty \,\sum_{\gamma\in K} \,\| g(\gamma) \|  + \| \tilde{C}(G,\varphi) \|_{op}\|f_\lambda\|_\infty \sum_{\gamma\notin K} \,\| g(\gamma) \| \\
&\leq& \|\tilde{C}(G,\varphi)\|_{op}(\|g\|_1+2)\|a\|\epsilon
\end{eqnarray*}
We conclude that $\nu_\lambda(g)\rightarrow \rho(\chi,a)(g)$.
\endproof

\subsection{Proof of Theorem \ref{t:twistedflex}}

\proof 
By Lemmas \ref{l:twistedflex2} and \ref{l:twistedflex3} we have,
$\nu_\lambda\stackrel{w^\ast}{\rightarrow}0$ and $\nu_\lambda\stackrel{w^\ast}{\rightarrow}\rho(\chi,a)$.
Since the $w^\ast$-topology is Hausdorff it follows that $\rho(\chi,a)=0$.
Thus the function $f_{\chi ,a}\in\ell^\infty(\Gamma, X^{V_0})$ given by,
\[f_{\chi,a}(\gamma) = T_{\tilde{\tau}}^{-1}(\chi \otimes a)(\gamma)
= (\chi (\gamma) d\tau(\gamma)a_{[v]})_{[v]\in V_0}
\]
lies in the kernel of $\tilde{C}(G,\varphi)$.
The result now follows since $z(\chi,a)=S_V^{-1}(f_{\chi ,a})$.
\endproof

The  {\em Rigid Unit Mode (RUM) spectrum} of $\G$ is defined as follows,
\[\Omega(\G) = \{\chi\in\hat{\Gamma} : \ker \Phi(\chi) \not=\{0\}\}.\]

\begin{remark}
The study of rigid unit modes and the RUM spectrum was initiated in \cite{gdph} as a means of understanding phase-transitions and structural stability in minerals. 
An operator-theoretic formulation of these notions was introduced by Owen and Power  in the context of periodic bar-joint frameworks in Euclidean space $\bR^d$ (\cite{owe-pow}).
In the above generalisation, characters $\chi$ in the dual group $\hat{\Gamma}$ can  be thought of as {\em wave vectors} in {\em reciprocal space}.
The $\chi$-symmetric vectors $z(\chi,a)$ which lie in the kernel of $C(G,\varphi)$ correspond to generalised {\em rigid unit modes} for the symmetric framework. 
\end{remark}


\section{Examples from discrete geometry}
\label{s:examples}
In this section we present some contrasting examples of symmetric frameworks arising from systems of geometric constraints. In each case, the underlying geometric structure is provided by a simple undirected graph $G$, a normed linear space $X$ and an assignment $p:V\to X$ of points in $X$ to each vertex in $G$. We consider 1) Euclidean distance constraints for a bar-joint framework with screw axis symmetry, 2) a direction-length framework with both periodic and reflectional symmetry and 3) mixed-norm distance contraints for a finite bar-joint framework with symmetry group $C_{4h}$. Each vector in the kernel of the associated coboundary matrix $C(G,\varphi)$ represents an {\em infinitesimal} (or {\em first-order}) {\em flex} of the framework. We derive the symbol function $\Phi$, compute the RUM spectrum $\Omega(\G)$ and construct $\chi$-symmetric infinitesimal flexes (i.e.~generalised rigid unit modes) for these frameworks. 

\subsection{Bar-joint frameworks in $\bR^d$}
A  {\em bar-joint framework}  in $\bR^d$ is a pair $(G,p)$ consisting of a simple  undirected graph $G=(V,E)$ and a point $p=(p_v)_{v\in V}\in (\bR^d)^{V}$ with the property that $p_v\not=p_w$ whenever $vw\in E$.
For each pair $v,w\in V$, set $\varphi_{v,w}:\bC^d\to \bC$, $x\mapsto (p_v-p_w)\cdot x$ if $vw\in E$ and $\varphi_{v,w}=0$ otherwise.
Then the pair $(G,\varphi)$ is a framework (for the Hilbert spaces $\bC^d$ and $\bC$) in the sense of Section \ref{s:symbolfunctions}. 


Expressing each linear map $\varphi_{v,w}$ as a row vector we obtain the   {\em rigidity matrix} $R(G,p)$ with rows indexed by $E$ and columns indexed by 
$V\times \{1,\ldots,d\}$. The row entries for a given edge $vw\in E$ are as follows,
\begin{eqnarray*}
\bordermatrix{ & &   & (v,1) & \cdots & (v,d) & & & & (w,1) & \cdots & (w,d) & &    \cr
vw & \cdots & 0 & p^1_v-p^1_w & \cdots & p^d_v-p^d_w & 0 & \cdots & 0 & p^1_w-p^1_v & \cdots & p^d_w-p^d_v & 0 & \cdots  \cr
}.\end{eqnarray*}

We begin with a small example.

\begin{example}
\label{ex:framework1}
Let $G=(V,E)$ be a four cycle with vertex set $V=\{v_1,v_2,v_3,v_4\}$ and edge set 
$E=\{v_1v_2,v_2v_3,v_3v_4,v_4v_1\}$.
Let $p=(p_v)_{v\in V}\in (\mathbb{R}^2)^V$ where,
\[p_{v_1}=(0,0),\quad p_{v_2} = (1,0),\quad p_{v_3} = (0,1),\quad p_{v_4}=(1,1).\]
The bar-joint framework $(G,p)$ is illustrated in Figure 
\ref{fig:framework} together with an accompanying rigidity matrix $R(G,p)$. 

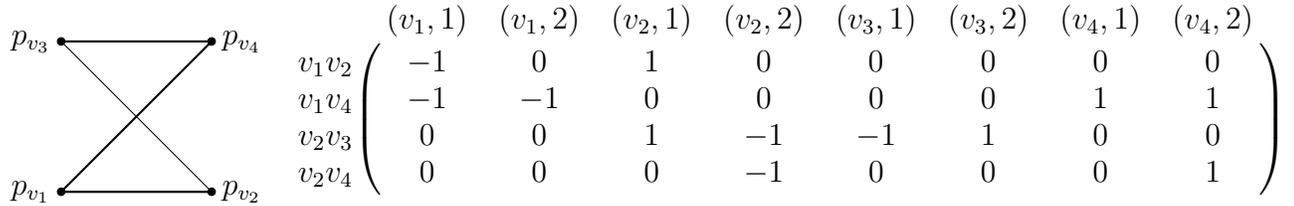
\begin{figure}[h!]
\centering
  \begin{tabular}{  c   }
  
\begin{minipage}{.3\textwidth}
\begin{tikzpicture}
\clip (-2,-0.4) rectangle (2.7cm,2.5cm);

\draw [thick](0,0)--(2,0);
\draw [thick](0,2)--(2,2);
\draw [fill] (0,2)--(2,0);
\draw [fill] (0,0) circle [radius=0.05];
\draw [fill] (2,0) circle [radius=0.05];
\draw [fill] (0,2) circle [radius=0.05];
\draw [fill] (2,2) circle [radius=0.05];
\draw [thick](0,0)--(2,2);
\draw [thick](0,0)--(2,0);
\node [left] at (0,0) {$p_{v_1}$};
\node [right] at (2,0) {$p_{v_2}$};
\node [left] at (0,2) {$p_{v_3}$};
\node [right] at (2,2) {$p_{v_4}$};
\end{tikzpicture}
\end{minipage}
  
\begin{minipage}{.7\textwidth}
\begin{eqnarray*}\bordermatrix{ 
& (v_1,1) & (v_1,2) & (v_2,1) & (v_2,2) & (v_3,1) & (v_3,2) & (v_4,1) & (v_4,2)\cr
v_1v_2  & -1 & 0 & 1 & 0 & 0 & 0 & 0 & 0 \cr
v_1v_4 & -1 & -1 & 0 & 0 & 0 & 0 & 1 & 1  \cr
v_2v_3 & 0 & 0 & 1 & -1 & -1 & 1 & 0 & 0  \cr
v_2v_4 & 0 & 0 & 0 & -1 & 0 & 0 & 0 & 1 \cr}\end{eqnarray*}
\vspace{4mm}
\end{minipage}
\end{tabular}
\caption{A bar-joint framework in $\mathbb{R}^2$ (left) and rigidity matrix (right).}\label{fig:framework2}
\end{figure}

Let $\theta:\bZ_2\to\Aut(G)$ be the group homomorphism described in Example \ref{ex:symgraph}. Let $\tau:\bZ_2\to \Isom(\bR^2)$ be the group homomorphism for which $\tau(1)$ is the orthogonal reflection in the line $y=\frac{1}{2}$.
Then $\G=(G,\varphi,\theta,\tau)$ is a $\bZ_2$-symmetric framework.
With the notation of Example \ref{ex:symframework}, the symbol function for $\G$ satisfies, 
\begin{eqnarray*}
\Phi(\chi_0) =\bordermatrix{ 
& ([v_1],1) & ([v_1],2) & ([v_2],1) & ([v_2],2)\cr
\left[e_1\right] & -1 & 0 & 1 & 0  \cr
\left[e_2\right] & -1 & -1 & 1 & -1 \cr},\end{eqnarray*}
\begin{eqnarray*}
\Phi(\chi_1)=\bordermatrix{ 
& ([v_1],1) & ([v_1],2) & ([v_2],1) & ([v_2],2)\cr
\left[e_1\right] & -1 & 0 & 1 & 0  \cr
\left[e_2\right] & -1 & -1 & -1 & 1  \cr}.\end{eqnarray*}
The multiplication operator $M_\Phi$ takes the form
\[M_\Phi:\bC^4\oplus \bC^4\to \bC^2\oplus \bC^2,\quad
\begin{bmatrix} x \\ y \end{bmatrix}\mapsto 
\begin{bmatrix} \Phi(\chi_0) & 0 \\ 0 & \Phi(\chi_1) \end{bmatrix}
\begin{bmatrix} x \\ y \end{bmatrix}.\]
In particular, we obtain the block diagonalisation of the  
 rigidity matrix $R(G,p)$ noted in Corollary \ref{c:block},
\[R(G,p)\sim \begin{bmatrix} \Phi(\chi_0) & 0 \\ 0 & \Phi(\chi_1) \end{bmatrix}.\]
Note that $\Omega(\G)=\{\chi_0,\chi_1\}$. The $\chi_0$-symmetric infinitesimal flexes derive from fully symmetric motions of the framework and take the form,
\[z_{v_1} = \left(\begin{smallmatrix} a \\ b \end{smallmatrix}\right), \quad
z_{v_2} =  \left(\begin{smallmatrix} a \\ -b \end{smallmatrix}\right), \quad
z_{v_3} = \left(\begin{smallmatrix} a \\ -b \end{smallmatrix}\right), \quad
z_{v_4} = \left(\begin{smallmatrix} a \\ b \end{smallmatrix}\right),\]
where $a,b\in \bC$.
The $\chi_1$-symmetric infinitesimal flexes take the form,
\[z_{v_1} = \left(\begin{smallmatrix} a \\ b \end{smallmatrix}\right), \quad
z_{v_2} =  \left(\begin{smallmatrix} a \\ -b \end{smallmatrix}\right), \quad
z_{v_3} = \left(\begin{smallmatrix} -a \\ b \end{smallmatrix}\right), \quad
z_{v_4} = \left(\begin{smallmatrix} -a \\ -b \end{smallmatrix}\right).\]
\end{example}

\begin{figure}[h!]
\centering
  \begin{tabular}{ c }
    \begin{minipage}{.3\textwidth}
\begin{tikzpicture}
\clip (-2,-2.5) rectangle (2.8cm, 4.5cm);

\draw [thick](-1,0)--(1,0);
\draw [thick](-1,1)--(1,1);
\draw [thick](-1,2)--(1,2);
\draw [thick](-1,3)--(1,3);
\draw [thick](-1,4)--(1,4);
\draw [thick](-1,-1)--(1,-1);
\draw [thick](-1,-2)--(1,-2);
\draw [thick](-1,-2.8)--(-1,4.8);
\draw [thick](1,-2.8)--(1,4.8);

\draw [fill] (1,0) circle [radius=0.05];
\draw [fill] (-1,0) circle [radius=0.05];
\draw [fill] (1,1) circle [radius=0.05];
\draw [fill] (-1,1) circle [radius=0.05];
\draw [fill] (1,2) circle [radius=0.05];
\draw [fill] (-1,2) circle [radius=0.05];
\draw [fill] (1,3) circle [radius=0.05];
\draw [fill] (-1,3) circle [radius=0.05];
\draw [fill] (1,4) circle [radius=0.05];
\draw [fill] (-1,4) circle [radius=0.05];
\draw [fill] (1,-1) circle [radius=0.05];
\draw [fill] (-1,-1) circle [radius=0.05];
\draw [fill] (1,-2) circle [radius=0.05];
\draw [fill] (-1,-2) circle [radius=0.05];

\node [right] at (1,0) {\small{$v_{1,0}$}};
\node [left] at (-1,0) {\small{$v_{0,0}$}};
\node [right] at (1,1) {\small{$v_{1,1}$}};
\node [left] at (-1,1) {\small{$v_{0,1}$}};
\node [right] at (1,2) {\small{$v_{1,2}$}};
\node [left] at (-1,2) {\small{$v_{0,2}$}};
\node [right] at (1,3) {\small{$v_{1,3}$}};
\node [left] at (-1,3) {\small{$v_{0,3}$}};
\node [right] at (1,4) {\small{$v_{1,4}$}};
\node [left] at (-1,4) {\small{$v_{0,4}$}};
\node [right] at (1,-1) {\small{$v_{1,-1}$}};
\node [left] at (-1,-1) {\small{$v_{0,-1}$}};
\node [right] at (1,-2) {\small{$v_{1,-2}$}};
\node [left] at (-1,-2) {\small{$v_{0,-2}$}};
\end{tikzpicture}
\end{minipage}
  
    \begin{minipage}{.3\textwidth}
\begin{tikzpicture}
\clip (-2.8,-2.78) rectangle (3.5cm, 4.5cm);

\draw [thick](-1,0,0)--(1,0,0);

\draw [red, thick](-0.707,1,-0.707)--(-1,0,0);
\draw [green, thick](0.707,1,0.707)--(1,0,0);
\draw [thick](0.707,1,0.707)--(-0.707,1,-0.707);
\draw [red, thick](0,2,-1)--(-0.707,1,-0.707);
\draw [green, thick](0,2,1)--(0.707,1,0.707);
\draw [thick](0,2,-1)--(0,2,1);
\draw [red, thick](0,2,-1)--(0.707,3,-0.707);
\draw [green, thick](0,2,1)--(-0.707,3,0.707);
\draw [thick](0.707,3,-0.707)--(-0.707,3,0.707);
\draw [red, thick](1,4,0)--(0.707,3,-0.707);
\draw [green, thick](-1,4,0)--(-0.707,3,0.707);
\draw [thick](1,4,0)--(-1,4,0);
\draw [red, thick](1,4,0)--(0.707,5,0.707);
\draw [green, thick](-1,4,0)--(-0.707,5,-0.707);

\draw [red, thick](-0.707,-1,-0.707)--(-1,0,0);
\draw [green, thick](0.707,-1,0.707)--(1,0,0);
\draw [thick](0.707,-1,0.707)--(-0.707,-1,-0.707);
\draw [green, thick](0,-2,1)--(0.707,-1,0.707);
\draw [red, thick](0,-2,-1)--(-0.707,-1,-0.707);
\draw [thick](0,-2,-1)--(0,-2,1);
\draw [red, thick](0,-2,-1)--(0.707,-3,-0.707);
\draw [green, thick](0,-2,1)--(-0.707,-3,0.707);

\draw [fill] (1,0,0) circle [radius=0.05];
\draw [fill] (-1,0,0) circle [radius=0.05];

\draw [thick,fill=white](0.707,1,0.707) circle [radius=0.05];
\draw [thick,fill=white](-0.707,1,-0.707) circle [radius=0.05];
\draw [thick,fill=white](0.707,-1,0.707) circle [radius=0.05];
\draw [thick,fill=white](-0.707,-1,-0.707) circle [radius=0.05];

\draw [thick,fill=white](0,2,1) circle [radius=0.05];
\draw [thick,fill=white](0,2,-1) circle [radius=0.05];
\draw [thick,fill=white](0,-2,1) circle [radius=0.05];
\draw [thick,fill=white](0,-2,-1) circle [radius=0.05];

\draw [thick,fill=white](-0.707,3,0.707) circle [radius=0.05];
\draw [thick,fill=white](0.707,3,-0.707) circle [radius=0.05];
\draw [thick,fill=white](-0.707,-3,0.707) circle [radius=0.05];
\draw [thick,fill=white](0.707,-3,-0.707) circle [radius=0.05];

\draw [thick,fill=white](1,4,0) circle [radius=0.05];
\draw [thick,fill=white](-1,4,0) circle [radius=0.05];

\node [below right] at (1,0,0) {\small{$p_{0,0}$}};
\node [above right] at (1,0,0) {\tiny{$(1,0,0)$}};
\node [below left] at (-1,0,0) {\small{$p_{1,0}$}};
\node [above left] at (-1,0,0) {\tiny{$(-1,0,0)$}};

\node [right] at (0.707,1,0.707) {\small{$p_{0,1}$}};
\node [above right] at (0.707,1.1,0.707) {\tiny{$(\frac{\sqrt{2}}{2},\frac{\sqrt{2}}{2},1)$}};
\node [left] at (-0.707,1,-0.707) {\small{$p_{1,1}$}};

\node [right] at (0.707,-1,0.707) {\small{$p_{0,-1}$}};
\node [left] at (-0.707,-1,-0.707) {\small{$p_{1,-1}$}};

\node [right] at (0,2,-1) {\small{$p_{(1,2)}$}};
\node [left] at (0,2,1) {\small{$p_{0,2}$}};

\node [below] at (0,0,0) {\small{$e_{1,0}$}};
\node [left] at (0.8,0.4,0) {\small{$e_{2,0}$}};
\node [right] at (-0.8,0.5,0) {\small{$e_{3,0}$}};
\end{tikzpicture}
\end{minipage}

    \begin{minipage}{.3\textwidth}
\begin{tikzpicture}
\clip (-2,-0.4) rectangle (3cm,2.4cm);

\draw [thick,->-] (0,0.8)--(2,0.8) node[midway,below] {\tiny{$0$}};
\draw [thick,->-] (0,0.8)to [out=45,in=0] node[right] {\tiny{$1$}}(0,1.5);
\draw [thick] (0,1.5)to [out=180,in=135](0,0.8);
\draw [thick,->-] (2,0.8)to [out=45,in=0] node[right] {\tiny{$1$}}(2,1.5);
\draw [thick] (2,1.5)to [out=180,in=135](2,0.8);
\draw [fill] (0,0.8) circle [radius=0.05];
\draw [fill] (2,0.8) circle [radius=0.05];
\node [left] at (0,0.8) {\small{$[v_{0,0}]$}};
\node [right] at (2,0.8) {\small{$[v_{1,0}]$}};
\end{tikzpicture}
\end{minipage}
\end{tabular}

\caption{The double helix framework $\G_{dh}$ (center), underlying graph (left) and gain graph (right).}\label{DH framework}
\end{figure}
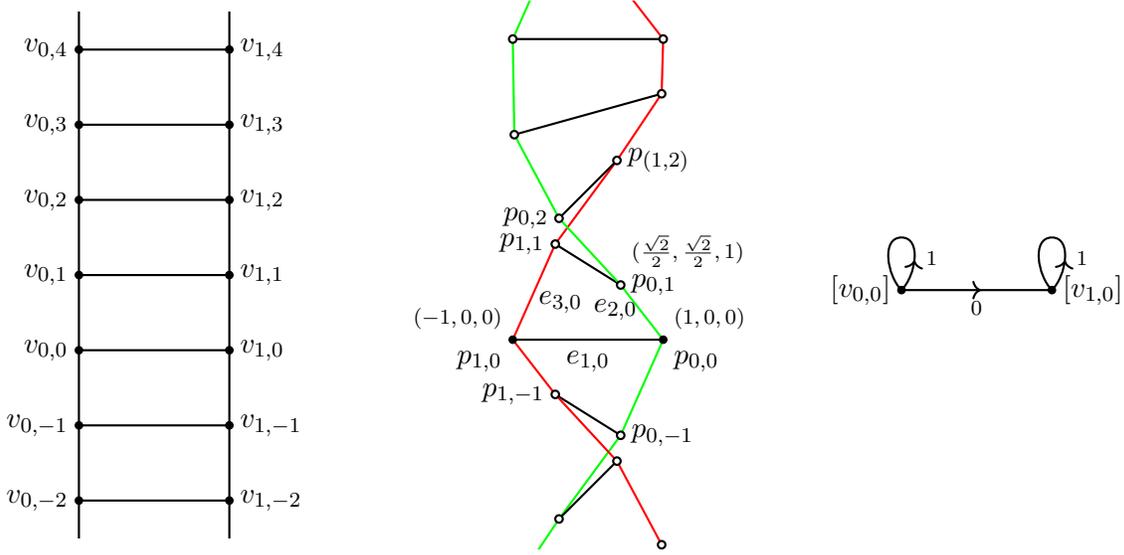

We now present our first main example.

\renewcommand{\arraystretch}{1.3}
\begin{example}{(Double helix framework)}
Consider the bar-joint framework $(G_{dh},p)$ in $\bR^3$, illustrated in Figure \ref{DH framework}. The  graph $G_{dh}$ has vertex set $V=\{v_{j,k}\,:\,j\in\{0,1\},k\in\bZ\}$ and edge set $E=\{e_{j,k}:\,j\in\{1,2,3\},k\in\bZ\}$ where  $e_{1,k}=v_{0,k}v_{1,k}$, $e_{2,k}=v_{0,k}v_{0,k+1}$ and $e_{3,k}=v_{1,k}v_{1,k+1}$. 
The placement $p:V\to \bR^3$ is defined by setting,
\[p_{j,k}:=p(v_{j,k})=\left( \begin{smallmatrix}(-1)^{j}\cos\left(\frac{k\pi}{4}\right)\\(-1)^{j}\sin\left(\frac{k\pi}{4}\right)\\k\end{smallmatrix}\right),\quad \forall\,j\in\{0,1\},\,k\in \bZ.\] 
Let $\theta:\bZ\rightarrow \operatorname{Aut}(G_{dh})$ be the group homomorphism with,
\[\theta(n)(v_{j,k})=v_{j,k+n},\quad
\forall\,j\in\{0,1\},\,k\in\bZ.\]
The quotient graph for the $\bZ$-symmetric graph $(G_{dh},\theta)$ is the multigraph $G_0=(V_0,E_0)$, where $V_0=\{[v_{0,0}], [v_{1,0}]\}$ is the set of vertex orbits and $E_0=\{[e_{1,0}],[e_{2,0}],[e_{3,0}]\}$ is the set of edge orbits.  
Choosing $v_{0,0}$ and $v_{1,0}$ as our vertex orbit representatives and fixing an orientation on the edges of $G_0$ we obtain a gain graph, such as the one shown in Figure \ref{DH framework}. 
Let $\tau:\bZ\rightarrow \operatorname{Isom}(\bR^3)$ be the group homomorphism which assigns to each $n\in\bZ$ the affine isometry $\tau(n)$ with linear part, 
\begin{equation*}
d\tau(n)=\left( \begin{smallmatrix}
\cos\left(\frac{n\pi}{4}\right) & -\sin\left(\frac{n\pi}{4}\right) & 0  \\
\sin\left(\frac{n\pi}{4}\right) & \cos\left(\frac{n\pi}{4}\right) & 0 \\ 
0 & 0 & 1
\end{smallmatrix}\right)
\end{equation*}
and translation vector $(0,0,n)\in\bR^3$. 
Note that, for each $n\in \bZ$, $\tau(n)$ is a screw rotation about the $z$-axis by the angle $\frac{\pi n}{4}$ and satisfies,
\[\tau(n)(p_{j,k})=p(\theta(n)(v_{j,k}))=p(v_{j,k+n})=p_{j,k+n},\quad \forall\,j\in\{0,1\},\,k\in\bZ.\]

Consider the $\bZ$-symmetric framework $\G_{dh}=(G_{dh},\varphi,\theta,\tau)$. 
To formulate the symbol function for $\G_{dh}$ we first compute,
\[p_{0,0}-p_{1,0}= \left(\begin{smallmatrix} 2 \\ 0 \\ 0  \end{smallmatrix}\right),\quad
p_{0,0}-p_{0,1}=\left(\begin{smallmatrix} 1-\frac{\sqrt{2}}{2} \\ -\frac{\sqrt{2}}{2}\\ -1  \end{smallmatrix}\right),\quad 
p_{1,0}-p_{1,1}=\left(\begin{smallmatrix} \frac{\sqrt{2}}{2}-1 \\ \frac{\sqrt{2}}{2}\\ -1  \end{smallmatrix}\right).\]
Recall that the dual group of $\mathbb{Z}$ consists of characters of the form 
$\chi_{\omega}:\mathbb{Z}\to \bT$, $k\mapsto \omega^k$, where $\omega\in \bT$. Thus, by Theorem \ref{t:symbol}, the symbol function $\Phi:\bT\rightarrow M_{3\times 6}(\bC)$ is given by,
\begin{eqnarray*}
\Phi(\omega) = 
\resizebox{0.85\hsize}{!}{
\bordermatrix{ &([v_{0,0}],1)&([v_{0,0}],2)&([v_{0,0}],3)&([v_{1,0}],1)& ([v_{1,0}],2) & ([v_{1,0}],3)\cr
([v_{0,0}],[v_{1,0}])&2 & 0 & 0 & -2 & 0 & 0  \cr  
([v_{0,0}],[v_{0,0}])&1-\frac{\sqrt{2}}{2}(1+\omega) & \omega-\frac{\sqrt{2}}{2}(1+\omega) & \omega-1 & 0 & 0 & 0\cr 
([v_{1,0}],[v_{1,0}])&0 & 0 & 0 & \frac{\sqrt{2}}{2}(1+\omega)-1 & \frac{\sqrt{2}}{2}(1+\omega)-\omega & \omega-1 \cr}
}\end{eqnarray*}

Note that $\Phi(\omega)$ has a 3-dimensional kernel for all $\omega\in \mathbb{T}$ and so $\Omega(\G_{dh})=\mathbb{T}$.

Calculating now the Fourier transform of $\Phi$, we obtain $\hat{\Phi}:\bZ\rightarrow M_{3\times 6}(\bC)$ where,
\[\hat{\Phi}(k) = \int_{\bT} \,\omega^{-k}\Phi(\omega)\,d\omega
= \left\{\begin{array}{ll} 
     {\tiny \begin{pmatrix} 2 & 0 & 0 & -2 & 0 & 0  \\ 
		1-\frac{\sqrt{2}}{2} & -\frac{\sqrt{2}}{2} & -1 & 0 & 0 & 0\\ 
0 & 0 & 0 & \frac{\sqrt{2}}{2}-1 & \frac{\sqrt{2}}{2} & -1 \end{pmatrix}}, & \text{if }k=0 \\
          {\tiny \begin{pmatrix} 0 & 0 & 0 & 0 & 0 & 0  \\ 
		-\frac{\sqrt{2}}{2} & 1-\frac{\sqrt{2}}{2} & 1 & 0 & 0 & 0\\ 
0 & 0 & 0 & \frac{\sqrt{2}}{2} & \frac{\sqrt{2}}{2}-1 & 1 \end{pmatrix}}, & \text{if }k=1 \\
       {\bf 0}_{3\times 6} , & \text{otherwise}. 
	\end{array}\right.\]
Then $\Phi(\omega) = \hat{\Phi}(0) + \hat{\Phi}(1)\omega$, as expected by Corollary \ref{cor:FCoef}.

Given any $\omega\in\bT$, it is easily checked that the vector 
$a=(1,-1,1,1,-1,-1)^T$ lies in the kernel of $\Phi(\omega)$.
Thus, by Theorem \ref{t:twistedflex}, the function 
\[z(\chi_\omega,a):V\to \bC^3,\quad v_{j,k}\mapsto 
\omega^{k}\begin{pmatrix}
\cos(\frac{k\pi}{4})+\sin(\frac{k\pi}{4})\\ \sin(\frac{k\pi}{4})-\cos(\frac{k\pi}{4})\\ (-1)^j\end{pmatrix},
\quad j\in\{0,1\},\,k\in\bZ.\] 
is a $\chi_\omega$-symmetric infinitesimal flex of the double helix framework.
\end{example}

\subsection{Direction-length frameworks}
A {\em direction-length framework} in $\bR^d$ is a pair $(G,p)$ consisting of a simple  undirected graph $G=(V,E)$, a partition of the edge set $E$ into two subsets $D$ and $L$, and a point $p=(p_v)_{v\in V}\in (\bR^d)^{V}$ with the property that $p_v\not=p_w$ whenever $vw\in E$.
For each pair $v,w\in V$, set $\varphi_{v,w}:\bC^d\to \bC^{d-1}$ to be,
\begin{enumerate}[(i)]
\item a linear map with rank $d-1$ and kernel spanned by $p_v-p_w$, if $vw\in D$, 
\item the linear map $x\mapsto ((p_v-p_w)\cdot x) I_{d-1}$, if $vw\in L$, and,
\item $0$, if $vw\notin E$.
\end{enumerate}
Note that the pair $(G,\varphi)$ is a framework (for the Hilbert spaces $\bC^d$ and $\bC^{d-1}$) in the sense of Section \ref{s:symbolfunctions}. 
The edges in $D$ represent direction constraints and the edges in $L$ represent length constraints. Mixed constraint systems of this type arise  naturally in CAD and network localisation for example (see \cite{ser-whi, jac-jor}).

\begin{example}{(Diamond lattice framework)}
Consider the diamond lattice direction-length framework illustrated in Figure \ref{ZZ2 framework}. 
The graph $G_{dl}$ has vertex set $V=\{v_{n,j}\,:\,n\in\bZ, \,j\in\{0,1\}\}$ and edge set  $E=D\cup L$ where 
$D=\{v_{n,j}v_{n+1,j}:\,n\in\bZ,\, j\in\{0,1\}\}$
and $L = \{v_{n,0}v_{n+1,1},v_{n,0}v_{n-1,1}:\,n\in\bZ,\, j\in\{0,1\}\}$.
The placement $p$ of $G_{dl}$ in $\bR^2$ satisfies $p_{n,j}:=p(v_{n,j})=(n,(-1)^{j+1})$ for all $n\in \bZ$ and $j\in\bZ_2$.

Given $v,w\in V$, define $\varphi_{v,w}:\bC^2\to \bC$ by setting,
\begin{enumerate}[(i)]
\item $\varphi_{v,w}(x_1,x_2)= x_2$ if $vw\in D$ is an edge with $v=v_{n,0}$ and $w=v_{n+1,0}$, or,  $v=v_{n+1,1}$ and $w=v_{n,1}$,
\item $\varphi_{v,w}(x_1,x_2)=-x_2$ if $vw\in D$ is an edge with $v=v_{n,1}$ and $w=v_{n+1,1}$, or,  $v=v_{n+1,0}$ and $w=v_{n,0}$,
\item $\varphi_{v,w}(x)=(p_v-p_w)\cdot x$ if $vw\in L$, and, 
\item $\varphi_{v,w}=0$ if $vw\notin E$.
\end{enumerate}
Then $(G,\varphi)$ is a framework (for the Hilbert spaces $\bC^2$ and $\bC$) in the sense of Section \ref{s:symbolfunctions}. 

Define a group homomorphism $\theta:\bZ\times \bZ_2\rightarrow \operatorname{Aut}(G_{dl})$ with, 
\[\theta(m,j)(v_{n,k})=v_{m+n,j+k},\quad m,n\in\bZ,\, j,k\in\bZ_2.\]
 Then the pair $(G_{dl},\theta)$ is a $\bZ\times \bZ_2$-symmetric graph. The accompanying gain graph $G_0=(V_0,E_0)$ has vertex set $V_0=\{[v_{0,0}]\}$ and edge set $E_0=\{[e_{1,(0,0)}],[e_{2,(0,0)}]\}$, where $e_{1,(0,0)}=v_{0,0}v_{1,0}$ and $e_{2,(0,0)}=v_{0,0}v_{1,1}$.

\begin{figure}[h!]
\centering
  \begin{tabular}{ c }
    \begin{minipage}{0.8\textwidth}
\begin{tikzpicture}
\clip (-5,-1.5) rectangle (7cm, 4cm);

\draw [thick](-7,0)--(-0,0);
\draw [thick](2,0)--(7,0);
\draw [thick](-7,2)--(2,2);
\draw [thick](2,2)--(7,2);
\draw [thick](-7,1)--(-6,2)--(-4,0)--(-2,2)--(0,0);
\draw [thick](2,2)--(4,0)--(6,2)--(7,1);
\draw [thick](-7,1)--(-6,0)--(-4,2)--(-2,0)--(0,2)--(2,0)--(4,2)--(6,0)--(7,1);
\draw [thick,fill=white] (-6,0) circle [radius=0.05];
\draw [thick,fill=white] (-4,0) circle [radius=0.05];
\draw [thick,fill=white] (-2,0) circle [radius=0.05];
\draw [red, thick](0,0)--(2,2);
\draw [red, thick](0,0)--(2,0);
\draw [red, fill] (0,0) circle [radius=0.05];
\draw [red, thick,fill=white] (2,0) circle [radius=0.05];
\draw [thick,fill=white] (4,0) circle [radius=0.05];
\draw [thick,fill=white] (6,0) circle [radius=0.05];
\draw [thick,fill=white] (-6,2) circle [radius=0.05];
\draw [thick,fill=white] (-4,2) circle [radius=0.05];
\draw [thick,fill=white] (-2,2) circle [radius=0.05];
\draw [thick,fill=white] (0,2) circle [radius=0.05];
\draw [red, thick,fill=white] (2,2) circle [radius=0.05];
\draw [thick,fill=white] (4,2) circle [radius=0.05];
\draw [thick,fill=white] (6,2) circle [radius=0.05];
\node [above] at (0,0.2) {\tiny{$(0,-1)$}};
\node [above] at (2,0.2) {\tiny{$(1,-1)$}};
\node [below] at (2,1.9) {\tiny{$(1,1)$}};
\node [below] at (0,0) {$p_{0,0}$};
\node [below] at (2,0) {$p_{1,0}$};
\node [above] at (1,-0.2) {$e_{1,(0,0)}$};
\node [left] at (1,1) {$e_{2,(0,0)}$};
\node [above] at (2,2) {$p_{1,1}$};
\end{tikzpicture}
\end{minipage}

 \begin{minipage}{0.2\textwidth}
\begin{tikzpicture}
\clip (-1,-1) rectangle (2cm,2.4cm);

\draw [thick,->-] (0,0.8)to [out=45,in=0] node[right] {\tiny{$(1,0)$}}(0,1.5);
\draw [thick] (0,1.5)to [out=180,in=135](0,0.8);
\draw [thick,->-] (0,0.8)to [out=-45,in=0] node[right] {\tiny{$(1,1)$}}(0,0.1);
\draw [thick] (0,0.1)to [out=180,in=-135](0,0.8);
\draw [fill] (0,0.8) circle [radius=0.05];
\node [left] at (0,0.8) {\small{$[v_{0,0}]$}};
\end{tikzpicture}
\end{minipage}
\end{tabular}
\caption{The diamond lattice direction-length framework $\G_{dl}$ (left) and its gain graph (right).}\label{ZZ2 framework}
\end{figure}
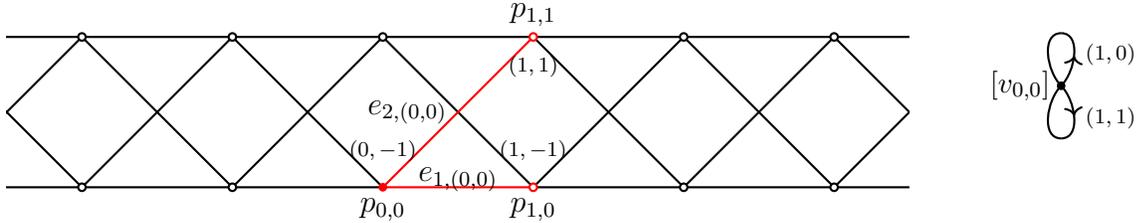

Define a group homomorphism $\tau:\bZ\times \bZ_2\to \Isom(\bR^2)$ with linear part,
\begin{equation*}
d\tau(m,j)=\left(\begin{smallmatrix} 
1 & 0  \\
0 & (-1)^j 
\end{smallmatrix}\right), \quad
m\in\bZ,\, j\in \bZ_2.
\end{equation*}
and translation vector $\left(\begin{smallmatrix}1 \\ 0 \end{smallmatrix}\right)$.
Note that $\theta$ and $\tau$ satisfy,
\[\varphi_{(m,j)v,(m,j)w} = \varphi_{v,w}\circ\tau(-m,-j),\quad \forall\,
m\in\bZ,\,j\in\bZ_2,\,v,w\in V.\]
Thus $\G_{dl}=(G_{dl},\varphi,\theta,\tau)$  is a $\bZ\times\bZ_2$-symmetric framework.

Recall that the dual group of $\bZ\times \bZ_2$ consists of characters of the form 
$\chi_{\omega,\iota}:\bZ\times \bZ_2\to \bT$, $(n,j)\mapsto \omega^n \iota^j$, where $\omega\in \bT$ and $\iota\in\hat{\bZ}_2=\{-1,1\}$.
Applying again Theorem \ref{t:symbol}, we obtain the symbol function,
\begin{eqnarray*}
\Phi(\omega,\iota)=
\bordermatrix{
&([v_{0,0}],1)& ([v_{0,0}],2) \cr
\left[e_{1,(0,0)}\right] &  0 & 1-\omega  \cr 
\left[e_{2,(0,0)}\right] & -1+\omega\iota & -2(1+\omega \iota) \cr}\end{eqnarray*}
where $\omega\in \bT$ and $\iota\in \hat{\bZ}_2$.
Note that $\Omega(\G_{dl})=\{(1,1),(1,-1),(-1,-1)\}$. 
We now apply Theorem \ref{t:twistedflex} to construct the associated $\chi$-symmetric infinitesimal flexes of $\G_{dl}$.

\begin{itemize}
\item Let $\omega=1$ and $\iota=1$. Check that $a:=\left(\begin{smallmatrix}  1\\0 \end{smallmatrix}\right) \in \operatorname{ker}\Phi(1,1)$. Hence we obtain a $\chi_{1,1}$-symmetric infinitesimal flex $z(\chi_{1,1},a) =(z_v)_{v\in V}$ where, 
\[z_{v_{m,j}}= d\tau(m,j)a
= \left(\begin{smallmatrix} 1\\ 0 \end{smallmatrix}\right),
\quad m\in \bZ,\,j\in\bZ_2.\] 
Note that this is a {\em trivial} infinitesimal flex of $\G_{dl}$ describing translation along the $x$-axis.

\item Let $\omega=1$ and $\iota=-1$. Check that $a:=\left(\begin{smallmatrix}  0\\1 \end{smallmatrix}\right)\in \operatorname{ker}\Phi(1,-1)$. Hence we obtain a $\chi_{1,-1}$-symmetric infinitesimal flex $z(\chi_{1,-1},a) =(z_v)_{v\in V}$ where, 
\[z_{v_{m,j}}= (-1)^{j} d\tau(m,j)a
=\left(\begin{smallmatrix} 0\\1 \end{smallmatrix}\right),
\quad m\in \bZ,\,j\in\bZ_2.\] 
Note that this is a trivial infinitesimal flex of $\G_{dl}$ describing translation along the $y$-axis.

\item Let $\omega=-1$ and $\iota=-1$. Check that $a:=\left(\begin{smallmatrix} 1\\0 \end{smallmatrix}\right) \in \operatorname{ker}\Phi(-1,-1)$. Hence we obtain a $\chi_{-1,-1}$-symmetric infinitesimal flex $z(\chi_{-1,-1},a) =(z_v)_{v\in V}$ where, 
\[z_{v_{m,j}}=(-1)^{m}(-1)^j d\tau(m,j)a 
= \left(\begin{smallmatrix}  (-1)^{m+j}\\ 0\end{smallmatrix}\right),
\quad m\in \bZ,\,j\in\bZ_2.\] 
Note that this is a non-trivial infinitesimal flex of $\G_{dl}$.
\end{itemize}
\end{example}

\subsection{Norm distance constraints}
Let $X$ be a finite dimensional real normed linear space with unit ball $B$.
There exists a unique ellipsoid in $X$ of minimal volume which contains $B$, known as the L\"{o}wner ellipsoid for $B$ (see \cite[p.~82]{thompson}). The L\"{o}wner ellipsoid is the unit ball for a norm  which is derived from an inner product on  $X$. Let $X'$ denote the real linear space $X$ together with this inner product and let $X'_\bC$ denote the complexification of this real Hilbert space.  

A  {\em bar-joint framework}  in $X$ is a pair $(G,p)$ consisting of a simple  undirected graph $G=(V,E)$ and a point $p=(p_v)_{v\in V}\in X^{V}$ with the property that $p_v-p_w$ is a non-zero smooth point in $X$ whenever $vw\in E$.
For each pair $v,w\in V$, set $\varphi_{v,w}:X\to \bR$ where, 
\begin{equation}\label{eqvarphigennorm}
\varphi_{v,w}(x) 
= \lim_{t\to 0} \frac{1}{t}\left(\|p_v-p_w+tx\|-\|p_v-p_w\|\right),
\end{equation}
if $vw\in E$ and $\varphi_{v,w}=0$ if $vw\notin E$.
Each linear map $\varphi_{v,w}$ extends in the natural way to a linear map from $X'_\bC$ to $\bC$.
Thus the pair $(G,\varphi)$ is a framework (for the Hilbert spaces $X'_\bC$ and $\bC$) in the sense of Section \ref{s:symbolfunctions}. 

Note that if $\theta:\Gamma\to \Aut(G)$ and $\tau:\Gamma\to \Isom(X)$ are group homomorphisms which satisfy $p_{\gamma v} = \tau(\gamma)p_v$, for all $v\in V$ and all $\gamma\in\Gamma$,
then it is straightforward to check that, 
\[\varphi_{\gamma v,\gamma w} = \varphi_{v,w}\circ \tau(-\gamma), \quad\forall\,v,w\in V,\,\gamma\in\Gamma.\]
The isometry group $\Isom(X)$ is a subgroup of $\Isom(X')$ 
(see \cite[Corollary 3.3.4]{thompson}) and each isometry of $X'$ has a natural extension to an isometry of $X'_\bC$. Thus, regarding $\tau$ as a homomorphism into $\Isom(X'_\bC)$, we see that $\G=(G, \varphi, \theta,\tau)$ is a $\Gamma$-symmetric framework in the sense of Section \ref{s:symbolfunctions}.

\begin{example}($\ell_{2,q}^3$ distance constraints)
Let $\ell_{2,q}^3$, where $q\in(1,\infty)$, denote the  vector space $\mathbb{R}^3$ equipped with the smooth mixed $(2,q)$-norm in $\bR^3$ given by, 
\[ \|(x,y,z)\|_{2,q} = ((x^2+y^2)^{\frac{q}{2}}+|z|^q)^{\frac{1}{q}}.\]
Infinitesimal rigidity for non-symmetric bar-joint frameworks in these spaces has recently been studied in \cite{ckks}. In particular, it is shown there that the Lowner ellipsoid for the unit ball in $\ell_{2,q}^3$ is the Euclidean unit ball in $\bR^3$. Thus the associated complex Hilbert space is $\bC^3$. 

Consider the box kite bar-joint framework in $\ell_{2,q}^3$, illustrated in Figure \ref{Z4Z2 framework}.
The underlying graph $G_{bk}$ has vertex set $V=\{v_{n,j}\,:\,n\in\bZ_4,\,j\in\bZ_2\,\}$ and edge set $E=\{v_{n,0}v_{n+1,1},\,v_{n,0}v_{n-1,1},\,v_{n,j}v_{n+1,j}:\,n\in\bZ_4,\, j\in\bZ_2\,\}$. 
The placement  $p:V\rightarrow\bR^3$ satisfies, for $j\in\{0,1\}$, 
\[p_{0,j}:=\resizebox{0.1\hsize}{!}{$\begin{pmatrix}-2\\-2\\(-1)^{j+1}\end{pmatrix}$},\quad 
p_{1,j}:=\resizebox{0.1\hsize}{!}{$\begin{pmatrix}2\\-2\\(-1)^{j+1}\end{pmatrix}$},\quad
p_{2,j}:=\resizebox{0.1\hsize}{!}{$\begin{pmatrix}2\\2\\(-1)^{j+1}\end{pmatrix}$},\quad 
p_{3,j}:=\resizebox{0.1\hsize}{!}{$\begin{pmatrix}-2\\2\\(-1)^{j+1}\end{pmatrix}$}.\]

%

\begin{figure}[h!]
\centering
  \begin{tabular}{ c }
	    \begin{minipage}{.4\textwidth}
\begin{tikzpicture}
\clip (-1,-1) rectangle (8cm, 5cm);

\draw [thick](2,0)--(6,0);
\draw [thick](0,2)--(6,2);
\draw [thick](2,2)--(4,0)--(6,2);
\draw [thick](0,2)--(2,0)--(4,2)--(6,0);
\draw [red, thick](0,0)--(2,2);
\draw [red, thick](0,0)--(2,0);
\draw [thick] (0,2)to [out=60,in=120](6.5,2);
\draw [thick] (6.5,2)to [out=-60,in=45](6,0);
\draw [thick] (6,2)to [out=120,in=60](-0.5,2);
\draw [thick] (-0.5,2)to [out=-120,in=135](0,0);
\draw [thick] (0,2)to [out=30,in=180](3,2.7);
\draw [thick] (3,2.7)to [out=0,in=150](6,2);
\draw [thick] (0,0)to [out=-30,in=180](3,-0.7);
\draw [thick] (3,-0.7)to [out=0,in=-150](6,0);

\draw [red, fill] (0,0) circle [radius=0.05];
\draw [red, thick,fill=white] (2,0) circle [radius=0.05];
\draw [red, thick,fill=white] (2,2) circle [radius=0.05];
\draw [thick,fill=white] (6,0) circle [radius=0.05];
\draw [thick,fill=white] (4,0) circle [radius=0.05];
\draw [thick,fill=white] (4,2) circle [radius=0.05];
\draw [thick,fill=white] (6,2) circle [radius=0.05];
\draw [thick,fill=white] (0,2) circle [radius=0.05];
\node [below] at (0,0) {\tiny{$v_{0,0}$}};
\node [below] at (2,0) {\tiny{$v_{1,0}$}};
\node [below] at (4,0) {\tiny{$v_{2,0}$}};
\node [below] at (6,0) {\tiny{$v_{3,0}$}};
\node [below left] at (0.2,2) {\tiny{$v_{0,1}$}};
\node [above] at (2,2) {\tiny{$v_{1,1}$}};
\node [above] at (4,2) {\tiny{$v_{2,1}$}};
\node [below right] at (5.8,2) {\tiny{$v_{3,1}$}};
\node [above,red] at (1,-0.1) {\tiny{$e_{1,(0,0)}$}};
\node [left,red] at (0.7,0.6) {\tiny{$e_{2,(0,0)}$}};
\end{tikzpicture}
\end{minipage}

 \begin{minipage}{.3\textwidth}
\begin{tikzpicture}
\clip (-3,-1.5) rectangle (5cm, 4cm);
\draw [thick](2.2,0,0)--(0.2,2.2,0);
\draw [thick](0.2,0,0)--(2.2,2.2,0);
\draw [thick](0.2,0,0)--(0,0,2);
\draw [thick](0,0,2)--(0.2,2.2,0);
\draw [thick](0.2,0,0)--(0,2.2,2);
\draw [thick](0.2,2.2,0)--(0,2.2,2);
\draw [thick](2.2,0,0)--(2,0,2);
\draw [thick](2,0,2)--(2.2,2.2,0);
\draw [thick](2.2,0,0)--(2,2.2,2);
\draw [thick](2.2,2.2,0)--(2,2.2,2);
\draw [thick](0.2,0,0)--(2.2,0,0);
\draw [thick](2,0,2)--(0,2.2,2);
\draw [thick](0.2,2.2,0)--(2.2,2.2,0);
\draw [thick](0,2.2,2)--(2,2.2,2);
\draw [red,thick](0,0,2)--(2,0,2);
\draw [red,thick](0,0,2)--(2,2.2,2);
\draw [red, fill] (0,0,2) circle [radius=0.05];
\draw [red, thick,fill=white] (2,0,2) circle [radius=0.05];
\draw [red, thick,fill=white] (2,2.2,2) circle [radius=0.05];
\draw [thick,fill=white] (0.2,2.2,0) circle [radius=0.05];
\draw [thick,fill=white] (2.2,0,0) circle [radius=0.05];
\draw [thick,fill=white] (0.2,0,0) circle [radius=0.05];
\draw [thick,fill=white] (2.2,2.2,0) circle [radius=0.05];
\draw [thick,fill=white] (0,2.2,2) circle [radius=0.05];
\node [above left] at (0.1,0,2.2) {\small{$p_{0,0}$}};
\node [below] at (0,0,2) {\tiny{$(-2,-2,-1)$}};
\node [below] at (2,0,2) {\small{$p_{1,0}$}};
\node [above, right] at (2.2,0.1,0) {\tiny{$(2,2,-1)$}};
\node [right] at (2.2,2.1,0) {\tiny{$(2,2,1)$}};
\node [right] at (2,0,2) {\tiny{$(2,-2,-1)$}};
\node [below] at (2.5,0.04,0) {\small{$p_{2,0}$}};
\node [above, red] at (1.1,0,2.2) {\tiny{$e_{1,(0,0)}$}};
\node [above, red] at (0.4,0.8,0) {\tiny{$e_{2,(0,0)}$}};
\node [above] at (2.2,2.15,0) {\small{$p_{2,1}$}};
\node [above] at (0,2.15,0) {\small{$p_{3,1}$}};
\end{tikzpicture}
\end{minipage}

 \begin{minipage}{.3\textwidth}
\begin{tikzpicture}
\clip (-3,-1) rectangle (4cm,2.4cm);

\draw [thick,->-] (0,0.8)to [out=45,in=0] node[right] {\tiny{$(1,0)$}}(0,1.5);
\draw [thick] (0,1.5)to [out=180,in=135](0,0.8);
\draw [thick,->-] (0,0.8)to [out=-45,in=0] node[right] {\tiny{$(1,1)$}}(0,0.1);
\draw [thick] (0,0.1)to [out=180,in=-135](0,0.8);
\draw [fill] (0,0.8) circle [radius=0.05];
\node [left] at (0,0.8) {\small{$[v_{0,0}]$}};
\end{tikzpicture}
\end{minipage}
\end{tabular}
\caption{The box kite bar-joint framework $\G_{bk}$ (center), underlying graph (left) and gain graph (right).}\label{Z4Z2 framework}
\end{figure}
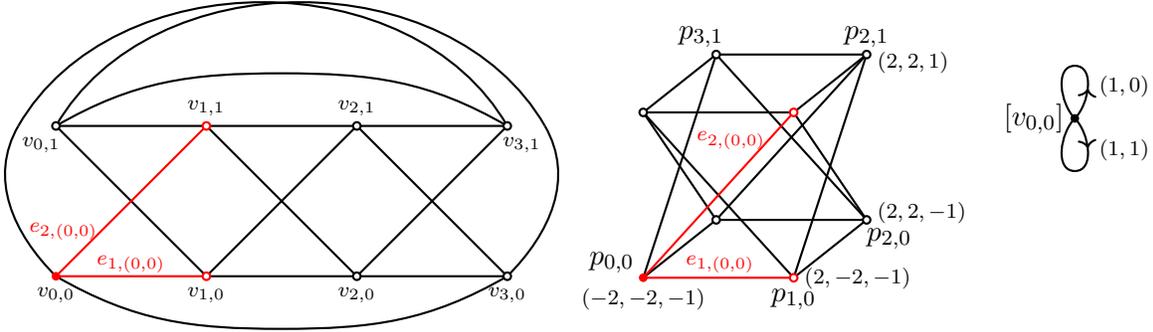

Define a group homomorphism $\theta:\bZ_4\times \bZ_2\rightarrow \operatorname{Aut}(G_{bk})$ with,
\[\theta(m,j)(v_{n,k})=v_{m+n,j+k},\quad \forall\, m,n\in\bZ_4,\,j,k\in\bZ_2.\]
 Then the pair $(G_{bk},\theta)$ is a $\bZ_4\times \bZ_2$-symmetric graph. The accompanying gain graph $G_0=(V_0,E_0)$ has vertex set $V_0=\{[v_{0,0}]\}$ and edge set $E_0=\{[e_{1,(0,0)}],[e_{2,(0,0)}]\}$, where $e_{1,(0,0)}=v_{0,0}v_{1,0}$ and $e_{2,(0,0)}=v_{0,0}v_{1,1}$. 

Define a group homomorphism $\tau:\bZ_4\times \bZ_2\rightarrow \Isom(\ell_{2,q}^3)$ with,
\begin{equation*}
\tau(m,j)=d\tau(m,j)=\left(\begin{smallmatrix}
\cos(m\pi/2) &-\sin(m\pi/2) & 0 \\
\sin(m\pi/2) & \cos(m\pi/2) & 0 \\ 0 & 0 & (-1)^j 
\end{smallmatrix}\right), 
\quad \forall\, m\in\bZ_4, \, j\in \bZ_2.
\end{equation*}
Note that,
\[p_{v_{m+n,j+k}} = \tau(m,j)p_{n,k},\quad \forall\, m,n\in\bZ_4, \, j,k\in \bZ_2.\]
Thus the tuple $\G_{bk}=(G_{bk}, \varphi, \theta,\tau)$ is a $\bZ_4\times \bZ_2$-symmetric framework (for the Hilbert spaces $(\ell_{2,q}^3)'_\bC$ and $\bC$).

Let now $vw\in E$. Write $p_v-p_w=(x,y,z)\in \ell_{2,q}^3$ and $d=\sqrt{x^2+y^2}$. Using the formula \eqref{eqvarphigennorm} we calculate directly,
\[\varphi_{v,w}(a,b,c)=(d^q+|z|^q)^{\frac{1}{q}-1}(d^{q-2}(xa+yb)+\operatorname{sgn}(z)|z|^{q-1}c),\quad \forall \,(a,b,c)\in \ell_{2,q}^3.\]
 Hence the functional $\varphi_{v,w}$ can be identified with the row vector
\[ \varphi_{v,w}=(d^q+|z|^q)^{\frac{1}{q}-1} d^{q-2}\begin{bmatrix} x & y & \frac{\operatorname{sgn}(z)|z|^{q-1}}{d^{q-2}}\end{bmatrix}.\]
The non-zero entries of the associated coboundary matrix are given by,
\[\varphi_{v_{0,0},v_{1,0}} = \begin{bmatrix} -1 & 0 & 0 \end{bmatrix},\quad
\varphi_{v_{0,0},v_{1,1}} = \alpha\begin{bmatrix} -2^{q-1} & 0 & -1\end{bmatrix},\]
\[\varphi_{v_{0,0},v_{3,0}} = \begin{bmatrix} 0 & -1 & 0\end{bmatrix},\quad
\varphi_{v_{0,0},v_{3,1}} = \alpha \begin{bmatrix} 0 & -2^{q-1} & -1\end{bmatrix},\]
where $\alpha=(2^q+1)^{\frac{1}{q}-1}$.

Recall that the dual group of $\bZ_4\times \bZ_2$ consists of characters of the form 
$\chi_{\eta,\iota}:\bZ_4\times \bZ_2\to \bT$, $(m,j)\mapsto \eta^m \iota^j$, where $\eta\in\hat{\bZ}_4=\{1,i,-1,-i\}$ and $\iota\in\hat{\bZ}_2=\{-1,1\}$.
By Theorem \ref{t:symbol}, the symbol function $\Phi:\hat{\bZ}_4 \times \hat{\bZ}_2\rightarrow M_{2\times 3}(\bC)$ of $\G_{bk}$ takes the form,
\begin{eqnarray*}\Phi(\eta,\iota)=
\bordermatrix{
&([v_{0,0}],1)& ([v_{0,0}],2) & ([v_{0,0}],3) \cr
\left[e_{1,(0,0)}\right] &  -1 & -\eta & 0 \cr 
\left[e_{2,(0,0)}\right] & -2^{q-1}\alpha & -2^{q-1}\alpha \eta\iota & 
-\alpha(1+\eta\iota) \cr}.\end{eqnarray*}
Evidently we have RUM spectrum $\Omega(\G_{bk})=\hat{\bZ}_4\times\hat{\bZ}_2$. 

First we will construct a $\chi_{1,1}$-symmetric infinitesimal flex of $\G_{bk}$. Note that such flexes represent a fully symmetric motion of the bar-joint framework which preserves the edge-lengths induced by the $(2,q)$-norm. 
The kernel of $\Phi(1,1)$ is spanned by $a=\left(\begin{smallmatrix} 1\\ -1 \\ 0 \end{smallmatrix}\right)$.
Thus, by Theorem \ref{t:twistedflex}, $z(\chi_{1,1},a)$ is a fully symmetric $\chi_{1,1}$-symmetric infinitesimal flex of $\G_{bk}$ where, for $j\in\bZ_2$,
\[z_{v_{0,j}} = \left(\begin{smallmatrix} 1\\ -1 \\ 0 \end{smallmatrix}\right),\quad z_{v_{1,j}}=\left(\begin{smallmatrix} 1\\ 1 \\ 0 \end{smallmatrix}\right),
\quad z_{v_{2,j}}=\left(\begin{smallmatrix} -1\\ 1 \\ 0 \end{smallmatrix}\right),
\quad z_{v_{3,j}}=\left(\begin{smallmatrix} -1\\ -1 \\ 0 \end{smallmatrix}\right).\]

Note that the above fully symmetric infinitesimal flex is independent of $q$. By way of constrast we now construct a $\chi_{-1,-1}$-symmetric infinitesimal flex for $\G_{bk}$ which varies with $q$.
Note that $\ker \Phi(-1,-1)$ is spanned by  $a=\left(\begin{smallmatrix} 1\\ 1 \\ -2^{q-1} \end{smallmatrix}\right)$. By Theorem \ref{t:twistedflex},  $z(\chi_{-1,-1},a)$ is a $\chi_{-1,-1}$-symmetric infinitesimal flex of $\G_{bk}$ where, for $j\in\bZ_2$,
\[z_{v_{0,j}} = \left(\begin{smallmatrix} 1\\ 1 \\ (-1)^{j+1}2^{q-1} \end{smallmatrix}\right),
\quad z_{v_{1,j}}=\left(\begin{smallmatrix} 1\\ -1 \\ (-1)^{j}2^{q-1} \end{smallmatrix}\right),
\quad z_{v_{2,j}}=\left(\begin{smallmatrix} -1\\ -1 \\ (-1)^{j+1}2^{q-1} \end{smallmatrix}\right),
\quad z_{v_{3,j}}=\left(\begin{smallmatrix} -1\\ 1 \\ (-1)^{j}2^{q-1} \end{smallmatrix}\right).\]
\end{example}

\end{document}